\documentclass[pdflatex,sn-mathphys-num]{sn-jnl}

\usepackage{graphicx}%
\usepackage{multirow}%
\usepackage{amsmath,amssymb,amsfonts}%
\usepackage{amsthm}%
\usepackage{mathrsfs}%
\usepackage[title]{appendix}%
\usepackage{xcolor}%
\usepackage{textcomp}%
\usepackage{manyfoot}%
\usepackage{booktabs}%
\usepackage{algorithm}%
\usepackage{algorithmicx}%
\usepackage{algpseudocode}%
\usepackage{listings}%
\usepackage{hyperref}
\usepackage[nameinlink]{cleveref}
\usepackage{tabularx}

\newcommand{\T}{\mathbb{T}}
\newcommand{\R}{\mathbb{R}}
\newcommand{\Z}{\mathbb{Z}}
\newcommand{\N}{\mathbb{N}}

\newcommand{\PP}{\mathcal{P}}

\newcommand{\W}{\mathcal{X}}

\newcommand{\LL}{\mathcal{L}}

\newcommand{\Xx}{\mathbf{x}}
\newcommand{\dXx}{\mathbf{\dot x}}
\newcommand{\Xxu}{\mathbf{x}_u}
\newcommand{\dXxu}{\mathbf{\dot x}_u}
\newcommand{\crit}{\sigma_u}
\newcommand{\cut}{\tau_u}

\DeclareMathOperator*{\argmin}{argmin} 
\DeclareMathOperator*{\argmax}{argmax} 
\DeclareMathOperator*{\supp}{spt}

\DeclareMathOperator*{\eps}{\varepsilon}

\DeclareMathOperator*{\mane}{\alpha[0]}

\DeclareMathOperator*{\Sing}{Sing}
\DeclareMathOperator*{\Crit}{Crit}
\DeclareMathOperator*{\Cut}{Cut}

\DeclareMathOperator{\co}{co}

\def\mane{\alpha[0]}
\def\maneV{\alpha[0]}

\theoremstyle{thmstyleone}%
\newtheorem{theorem}{Theorem}
\newtheorem{proposition}[theorem]{Proposition}
\newtheorem{lemma}[theorem]{Lemma}
\newtheorem{corollary}[theorem]{Corollary}%

\theoremstyle{thmstyletwo}%
\newtheorem{example}{Example}%
\newtheorem{remarks}{Remark}%

\theoremstyle{thmstylethree}%
\newtheorem{definition}{Definition}%

\begin{document}

\title[Gradient flows of solutions to Hamilton-Jacobi equations]{Long-time behavior of generalized gradient flows\\ of solutions to Hamilton-Jacobi equations}


\author[1]{\fnm{Paolo} \sur{Albano}}\email{paolo.albano@unibo.it}
\equalcont{These authors contributed equally to this work.}

\author*[2]{\fnm{Piermarco} \sur{Cannarsa}}\email{cannarsa@axp.mat.uniroma2.it}
\equalcont{These authors contributed equally to this work.}

\author[3]{\fnm{Wei} \sur{Cheng}}\email{chengwei@nju.edu.cn}
\equalcont{These authors contributed equally to this work.}

\author[4]{\fnm{Cristian} \sur{Mendico}}\email{cristian.mendico@u-bourgogne.fr}
\equalcont{These authors contributed equally to this work.}

\affil[1]{\orgdiv{Dipartimento di Matematica}, \orgname{Universit\`a di Bologna}, \orgaddress{\street{Piazza 
di Porta San Donato 5}, \city{Bologna}, \postcode{40127}, \state{Italy}}}

\affil*[2]{\orgdiv{Dipartimento di Matematica}, \orgname{Universit\`a di Roma "Tor Vergata''}, \orgaddress{\street{Via della Ricerca Scientifica 1}, \city{Roma}, \postcode{00133}, \state{Italy}}}

\affil[3]{\orgdiv{School of Mathematics}, \orgname{Nanjing University}, \orgaddress{\city{Nanjing}, \postcode{210093}, \state{China}}}

\affil[4]{\orgdiv{Institut de Math\'ematique de Bourgogne, UMR 5584 CNRS}, \orgname{Universit\'e Bourgogne Europe}, \orgaddress{\street{9 avenue Alain Savary}, \city{Dijon}, \postcode{21000}, \state{France}}}

\abstract{We study the long-time behavior of the generalized gradient flow associated with solutions of the critical Hamilton-Jacobi equation for mechanical Hamiltonians on the flat torus. For any semiconcave function, we show that its critical set---points whose superdifferential contains the zero vector---acts as an approximate attractor for the flow. When the function is a solution of the critical equation, the critical set decomposes into regular and singular parts, and we establish a dichotomy describing which part trajectories approach as $t \to \infty$. Our analysis uses limiting occupational measures, a class of invariant measures capturing the asymptotic distribution of the flow. An essential ingredient is a complete proof of the global invariance of the singular set, a result previously announced by Albano (2016) but not fully established.}

\keywords{Hamilton-Jacobi equations; Viscosity solutions; Singularities; Generalized characteristics; Occupational measures.}

\pacs[MSC Classification]{35F21, 35A21, 37J51, 49L25.}

\maketitle

\section{Introduction}\label{sec1}

The dynamics of generalized characteristics associated with semiconcave viscosity solutions of Hamilton-Jacobi equations have been extensively studied in the past two decades, particularly in the framework of weak KAM theory (see, e.g., \cite{Albano_Cannarsa1999,Albano_Cannarsa2002,Bogaevsky2002,Yu2006,MR2237158,Cannarsa_Yu2009,ACNS2013,Cannarsa_Cheng3,Albano2016_1,Cannarsa_Chen_Cheng2019,Cannarsa_Cheng2021b,Cannarsa_Cheng_Fathi2021,Chen_Hong_Zhao2022,Cannarsa_Cheng_Hong2023,Cannarsa_Cheng_Hong_Wang2024}). While a detailed understanding has been achieved for local propagation of singularities and for several classes of time-dependent or discounted equations, a systematic description of the long-time behavior of such flows remains incomplete, even in the stationary setting. In particular, it is still unclear how trajectories distribute asymptotically and to what extent they are attracted by distinguished sets such as the critical set or the projected Aubry set.

In this paper, we develop a new approach to these questions based on occupational measures, a classical tool in dynamical systems and control theory (see, e.g., \cite{Artstein1999}) that has not been previously exploited in this context. This framework allows us to capture the asymptotic behavior of generalized gradient flows in a measure-theoretic way, leading to new dynamical and structural results.

The second main contribution is a \emph{complete proof of global propagation of singularities} for the stationary critical mechanical Hamilton-Jacobi equation
\begin{equation*}
\frac12|Du(x)|^2 + V(x) = \alpha[0], \quad x \in \mathbb{T}^d,
\end{equation*}
with $V \in C^2(\mathbb{T}^d)$ and $\alpha[0] = \max V$. While Albano~\cite{Albano2016_1} announced such a result, the available argument left a gap in both the stationary and time-dependent case; our proof fills this gap for the former class of equations (Theorem~\ref{t:mono}). We also obtain a \emph{dichotomy result} (Proposition~\ref{pro:finite_critical}) describing the first contact between a generalized characteristic and the critical set.

Finally, in Section~\ref{sec:occupational}  we use occupational measures to characterize the asymptotic behavior of generalized characteristics in the critical case. We show that the limiting measures charge only the set $\{V=0\}$ if and only if the trajectory spends asymptotically all time near the projected Aubry set $\mathrm{Crit}^*(u) = \arg\max V$ (Theorem~\ref{the:case=0}). Complementary conditions lead to asymptotic concentration in the \emph{singular} part of the critical set (Theorem~\ref{the:case<0} and Corollary~\ref{cor:summa}). To our knowledge, these measure--theoretic characterizations are new.

Overall, the paper:
\begin{enumerate}
\item Introduces an \emph{occupational-measure framework} for generalized characteristics of semiconcave functions, yielding an approximate attractor property for the critical set.
\item Provides a \emph{complete proof} of global propagation in the stationary critical mechanical case.
\item Establishes \emph{new asymptotic characterizations} via invariant measures for the flow.
\end{enumerate}

We believe these results offer a robust link between dynamical systems methods and weak KAM theory, and can be further developed for more general Hamiltonians and non-stationary settings.

\smallskip
The structure of the paper is as follows. In Section~\ref{sec:semiconcave} we recall basic properties of semiconcave functions and their generalized gradient flows. Section~\ref{sec:GF&OM} introduces the occupational-measure framework and establishes the approximate attractor property for the critical set (Theorem~\ref{the:asympt_critical}), which is one of the main conceptual contributions of the paper. Section~\ref{sec:solutions} is devoted to the proof of global propagation of singularities for the stationary critical Hamilton-Jacobi equation (Theorem~\ref{t:mono}), which provides a complete and simplified resolution of a previously announced result. Finally, in Section~\ref{sec:occupational} we apply the occupational-measure approach to describe the asymptotic behavior of generalized characteristics, obtaining a dichotomy between convergence to the regular and singular parts of the critical set (Theorems~\ref{the:case=0} and~\ref{the:case<0}, and Corollary~\ref{cor:summa}).

\section{Gradient flows of semiconcave functions}
\label{sec:semiconcave}
We denote by $\T^d$ the flat torus. 
One relevant notion to the topics of this paper is that of a semiconcave function.
\begin{definition}\label{def:semiconcave}
We say that a continuous function $u: \T^d \to \R$ is {\it semiconcave} with linear modulus (or, briefly, semiconcave) if there exists a constant $C > 0$ such that 
\begin{equation*}
\lambda u(x) + (1-\lambda) u(y) - u(\lambda x + (1-\lambda) y) \leqslant   \frac{C}{2}\lambda(1-\lambda)|x-y|^{2}
\end{equation*}
for any $x$, $y \in \T^{d} $ and any $\lambda \in [0,1]$. 
\end{definition} 
Any semiconcave function $u: \T^d \to \R$  is Lipschitz continuous and possesses a nonempty super-differential $D^+u(x)$ for all $x\in\T^d$ (see, e.g., \cite{Cannarsa_Sinestrari_book}). We recall that, for any $x\in\T^d$, the generalized Dini semi-differentials of $u$ at $x$ are the closed convex sets defined by
\begin{align*}
D^-u(x)&=\left\{p\in\R^n:\liminf_{y\to x}\frac{u(y)-u(x)-\langle p,y-x\rangle}{|y-x|}\geqslant 0\right\}\\
D^+u(x)&=\left\{p\in\R^n:\limsup_{y\to x}\frac{u(y)-u(x)-\langle p,y-x\rangle}{|y-x|}\leqslant   0\right\}.
\end{align*}
The following sets will play a crucial role in this paper.
\begin{definition}\label{def:critical}
The {\em critical} set of $u$ is defined as 
$$\Crit(u) = \big\{x \in \T^{d}: 0 \in D^{+}u(x)\big\}.$$
The {\em singular} set of $u$ is defined as 
$$\Sing(u) = \{x \in \T^{d}: Du(x)\,\, \text{does not exist}\}.$$
The {\em regular critical} set of $u$ is given by 
\begin{equation*}
\Crit \!\!\,^{*}(u) = \Crit(u) \backslash \Sing(u).
\end{equation*}
\end{definition}
\begin{remarks}\label{rem:critical}
The following are basic properties of the above sets.
\begin{itemize}
\item[$(a)$] The critical set is {\em nonempty}: indeed $\Crit(u)$ contains the maximum points of $u$.
\item[$(b)$] The critical set  is {\em closed} because the set-valued map $D^+u:\T^d\rightrightarrows \R^d$ is upper semi-continuous (see, for instance, \cite[Proposition~3.3.4]{Cannarsa_Sinestrari_book}). 
\item[$(c)$] The regular critical set is {\em nonempty}: indeed $\Crit^*(u)$ contains the minimum points of $u$. For if $u$ attains a local minimum at $x\in\T^d$, then $0\in D^-u(x)$. Since $D^+u(x)\neq \varnothing$ because $u$ is semiconcave, we have that $u$ is differentiable at $x$ and $Du(x)=0$ by \cite[Proposition~3.1.5]{Cannarsa_Sinestrari_book}. 
\end{itemize}
\end{remarks}

\smallskip
In what follows, we will need the following direct consequence of  Alexandrov's theorem on the second order differentiability of convex functions (\cite[Theorem 1.2]{azagra-cappello-hajlasz_2023}). 
\begin{theorem}\label{alexandrov}
Let $u:\T^d \longrightarrow \R$ be a semiconcave function. Then $u$ is differentiable a.e. on  $ \T^d$  and, for a. e.   point $x\in \T^d$ at which $u$ is differentiable, there exists a symmetric matrix, denoted by $D^2u(x)$, such that  
\begin{equation}\label{eq:2diff}
	\lim_{y\to x}\sup_{p\in D^+u(y)} \frac{|p-Du(x)-D^2u(x)(y-x)|}{|y-x|}=0. 
\end{equation}
\end{theorem}
Observe that $x\mapsto D^2u(x)$ is measurable and bounded above as a quadratic form.
\subsection{Generalized gradient flow}
\label{ssec:ggf}
Let  
  $u: \T^d \to \R$ be a semiconcave function. It is well known that for $x \in \T^d$  the problem 
\begin{equation}\label{flow}
\begin{cases}
\dXx(t, x) \in D^{+}u(\Xx(t, x)), \quad t\in[0,\infty)\mbox{ a.e.}
\\
\Xx(0, x)=x.
\end{cases}
\end{equation}
has a unique solution (see, e.g., \cite{Aubin_cellina1984}). Such a solution, which is locally Lipschitz continuous in $(t,x)$ on $[0,\infty)\times\Omega$, is called the {\em generalized gradient flow} of $u$  and is denoted by
 $\Xxu(t, x)$.
In fact,  $\Xxu(t, x)$ is a semi-flow because uniqueness and regularity are guaranteed just for $t\geqslant 0$
(see, e.g.,  \cite[Lemma~4]{Albano_Cannarsa2002}). Clearly,  $\Crit(u)$ coincides with the set of all equilibrium points for $\Xxu$.

\smallskip
We collect below results for the generalized gradient flow of interest to this paper. 
\begin{proposition}\label{pro:ggf}
Let $x\in\T^d$. Then the following holds true.
\begin{itemize}
\item[$(a)$] $\Xxu(\cdot,x)$ has the right derivative $\dXxu^+(t,x)$ for all $t\in[0,\infty)$ and 
 \begin{equation}
\label{eq:ggf}
\dXxu^+(t, x)  = p_{0}\big(\Xxu(t, x)\big),
\end{equation}
where
\begin{equation}
\label{eq:ms}
p_{0}(y)= \argmin_{p \in D^{+}u(y)} |p|\qquad(y\in\T^d).
\end{equation}
\item [$(b)$] The right derivative of $u\big(\Xxu(\cdot, x)\big)$ exists for all $t\in[0,\infty)$ and is given by
\begin{equation}
\label{eq:dggf}
\frac{d^+}{dt} u\big(\Xxu(t, x)\big) =
\big|p_{0}\big(\Xxu(t,x)\big)\big|^{2}.
\end{equation}
\item[$(c)$] For a.e. $t\in [0,T]$ and  every $p\in D^+u (\Xxu (t,x))$, we have that 
	\begin{equation}\label{const}
p\cdot 	p_{0}\big(\Xxu(t,x)\big)=\big|p_{0}\big(\Xxu(t,x)\big)\big|^2.
	\end{equation} 
\end{itemize}
 \end{proposition}
\proof  For points $(a)$ and $(b)$ we refer the reader to \cite[Theorem~1]{ACNS2013}, where the conclusion is obtained for solutions of eikonal type equations on bounded domains. 
Since the argument needed to derive \eqref{eq:ggf} and \eqref{eq:dggf} is of local nature and just uses the differential inclusion in \eqref{flow}, the same proof applies to the present case.

As for point $(c)$, we use an argument taken from the proof of
\cite[Proposition 3.2]{Cannarsa_Cheng_Hong_Wang2024}. 
By the representation of directional derivatives (see \cite{Cannarsa_Sinestrari_book}) and the fact that $u\big(\Xxu(\cdot, x)\big)$ is Lipschitz, for  a.e. $t\in [0,T]$ we have that 
\begin{multline*}
\frac{du}{dt }(\Xxu (t,x))=\frac{d^+u }{dt}(\Xxu (t,x))=\min_{p \in D^+ u(\Xx (t,x))} p\cdot p_{0}\big(\Xxu(t,x)\big)
\\
= \frac{d^-u }{dt}(\Xxu (t,x))=-\min_{p \in D^+ u(\Xxu (t,x))} p\cdot \big(-p_{0}(\Xxu(t,x))\big)=\max _{p \in D^+ u(\Xxu (t,x)} p\cdot p_{0}\big(\Xxu(t,x)\big).
\end{multline*}
Therefore, the function $p\mapsto p\cdot p_{0}(\Xxu(t,x))$ is constant on $D^+ u(\Xxu (t,x))$. The conclusion follows. 
\qed
\begin{remarks}\label{re:p_0} The following properties of the map $p_0:\T^d\to\R^d$ in \eqref{eq:ms} are easily deduced.
\begin{enumerate}[(a)]
\item $x\mapsto p_0(x)$ is single-valued: indeed, $p_0(x)$ is the projection of $0$ onto the non-empty convex compact set $D^+u(x)$.
\item The function $x\mapsto|p_0(x)|$ is lower semi-continuous and, in particular, Lebesgue measurable: once again, this is a consequence of the upper semi-continuity of the set-valued map $D^+u:\T^d\rightrightarrows \R^d$.
\end{enumerate}
\end{remarks}
 \begin{corollary}\label{cor:ggf}
Let $x\in\T^d$. Then:
\begin{enumerate}
	\item[$(i)$] the function $t\mapsto u\big(\Xxu(\cdot,x)\big)$ is nondecreasing on $[0,\infty)$,
	\item[$(ii)$]  there exists a sequence  of positive real numbers $t_k\to\infty $ such that
	\begin{align*}
		p_{0}\big(\Xxu(t_k, x)\big)\to 0\quad\mbox{ as }\quad k\to\infty.
	\end{align*}
\end{enumerate}
\end{corollary}
\proof
Point $(i)$ is a direct consequence of \eqref{eq:dggf}. As for $(ii)$,  by integrating \eqref{eq:dggf} we obtain
\begin{equation*}
\int_0^\infty \big|p_{0}\big(\Xxu(t,x)\big)\big|^{2}dt\leqslant   2\|u\|_\infty.
\end{equation*}
The conclusion follows.
\qed
\begin{remarks}\label{re:omega-limit}
We observe that point $(ii)$ above yields that, for any $x\in\T^d$,  the $\omega$-limit set of $\Xxu(\cdot,x)$
intersects the critical set of $u$ or, equivalently,
\begin{equation}
\label{eq:omega_limit}
\liminf_{t\to\infty}d_{\Crit(u)}\big(\Xxu(t,x)\big)=0\,,\quad \forall x\in\T^d\,,
\end{equation}
where  $d_S(x)$ stands for the Euclidean distance of $x$ from the closed set $S$. In \cite[Theorem~2.5]{Cannarsa_Chen_Cheng2019} (see also \cite{Hurley1995}), for weak KAM solutions $u$ of Hamilton-Jacobi equations on the torus like the ones of interest to this paper, the $\omega$-limit set of $\Xxu(\cdot,x)$ was proven to actually coincide with $\Crit(u)$ provided that the regular values\footnote{We recall that $\lambda\in\R$ is a regular value of $u$ if $\{x\in\T^d~:~u(x)=\lambda\}\cap\Crit(u)=\varnothing$.} of $u$ are dense in $\R$. In this case, \eqref{eq:omega_limit} is replaced by the stronger property
\begin{equation}
\label{eq:Omega_limit}
\lim_{t\to\infty}d_{\Crit(u)}\big(\Xxu(t,x)\big)=0\,,\quad \forall x\in\T^d\,.
\end{equation}
Moreover, Morse-Sard type results from~\cite{Rifford2008} ensure  such a density property for $d=2$, or even $d=3$ when the Hamiltonian is sufficiently smooth. In the next section, we will obtain a relaxed version of \eqref{eq:Omega_limit} under no extra assumption on $u$ than semiconcavity.
\end{remarks}
\section{Occupational measures of a gradient flow}
\label{sec:GF&OM}
The purpose of this section is to introduce occupational measures as a tool to describe the asymptotic behavior of generalized gradient flows. Roughly speaking, these measures capture the time distribution of trajectories and allow us to obtain information on their long-time behavior without requiring pointwise convergence.

Let us denote by  $\PP(\T^d)$ the family of all probability measures on $\T^d$. We recall that a sequence $\{\mu_k\}_{k\in\N}\subset \PP(\T^d)$ {\em weakly converges} to a measure $\mu\in\PP(\T^d)$ (or, $\mu_k\rightharpoonup \mu$) if
\begin{equation*} 
\lim_{k\to\infty}\int_{\T^d} f(x) d\mu_k(x) =\int_{\T^d} f(x) d\mu(x) 
\qquad  \forall  f \in C(\T^d). 
\end{equation*}

Let  
 $u: \T^{d} \to \R$ be a semiconcave  function and let $\Xxu$ be the semi-flow associated with \eqref{flow}.
We recall that a  measure $\mu \in \PP(\T^{d})$ is {\em invariant under} $\Xxu$ (or, $\Xxu$-{\em invariant}) if
\begin{equation*}
 \int_{\T^{d}} 
f\big(\Xxu(t, x)\big)d\mu(x)=\int_{\T^{d}} f(x) d\mu(x) 
\qquad  \forall t\geqslant 0,\;\forall f \in C(\T^{d}). 
\end{equation*}
The following result provides a characterization of invariant measures for the generalized gradient flow.
\begin{proposition}\label{pro:cnes}
A measure $\mu \in \PP(\T^{d})$ is invariant under $\Xxu(t, \cdot)$ if and only if 
\begin{equation}
\label{eq:cnes}
\supp(\mu) \subset \Crit(u).
\end{equation}
\end{proposition}
\proof 
Let  $\mu \in \PP(\T^{d})$ satisfy \eqref{eq:cnes}. Then, for any  $f \in C(\T^{d})$ and $t\geqslant 0$ we have that 
\begin{equation*}
\int_{\T^{d}} f\big(\Xxu(t, x)\big)\ d\mu(x) = \int_{\Crit(u)} f\big(\Xxu(t, x)\big)\ d\mu(x)
\end{equation*}
Now, observe that $\Xxu(t,x)\equiv x$  for any $x \in \Crit(u)$ by uniqueness.  So,
\begin{equation*}
 \int_{\T^{d}} f\big(\Xxu(t, x)\big)\ d\mu(x) =\int_{\Crit(u)} f(x)\ d\mu(x)=\int_{\T^{d}} f(x) d\mu(x) .
\end{equation*}
 Thus, $\mu$ is $\Xxu$-invariant.

Conversely, let $\mu \in \PP(\T^d)$ be invariant under $\Xxu$. Then, \eqref{eq:dggf} ensures that
\begin{equation*}
0 = \frac{d}{dt} \int_{\T^{d}} u\big(\Xxu(t,x)\big)\ d\mu(x) = \int_{\T^{d}} \big|p_{0}\big(\Xxu\big(t,x)\big)\big|^{2}\ d\mu(x) \quad (t \geq 0\mbox{ a.e.})
\end{equation*}
Thus, for a.e. $t \geqslant 0$ we have that $p_{0}\big(\Xxu\big(t,\cdot)\big)=0$  on the support of $\mu$. In fact, $\big|p_{0}\big(\Xxu\big(t,\cdot)\big)|$ is lower semi-continuous owing to Remark~\ref{re:p_0}. So, $p_{0}\big(\Xxu\big(t,\cdot)\big)$ vanishes on the support of $\mu$ for all $t\geqslant 0$. In particular,  $p_{0}(x)\equiv 0$  on the support of $\mu$  and $\supp(\mu) \subset \Crit(u)$. \qed

\bigskip
For any $x \in \T^{d}$ and any $T > 0$ the measure $\mu^{T}_{x} \in \PP(\T^{d})$, defined by 
\begin{equation*}
\int_{\T^{d}} f(y)\ d\mu^{T}_{x}(y) = \frac{1}{T} \int_{0}^{T} f\big(\Xxu(t,x)\big)\ dt \quad \forall f \in C(\T^{d}),
\end{equation*}
is called the {\em individual occupational measure of $\Xxu(\cdot,x)$ on $[0,T]$}. Occupational measures have long been in use in the mathematical literature, see \cite{Artstein1999} for a systematic use of such measures in the study of differential inclusions.
\begin{definition}\label{def:wom}
We say that a measure $\mu \in \PP(\T^{d})$ is an {\em  occupational measure} from a point $x\in\T^d$ for the semi-flow $\Xxu$ if there exists a sequence $\{T_{k}\}_{k \in \N}$, with $T_{k} \uparrow \infty$,
such that 
\begin{equation}
\label{eq:wom}
\mu_{x}^{T_{k}} \rightharpoonup \mu\qquad(k\to \infty).
\end{equation}
We denote by $\W_u(x)$ the family of all   occupational measures from $x$ for  $\Xxu$.
\end{definition}
By repeating the proof of the classical Krylov-Bogoliubov theorem for semiflows (see e.g. \cite{KrylovBogolyubov1937,KatokHasselblatt,Petersen1983}), 
one ensures the existence of occupational measures and their invariance properties.
\begin{proposition}\label{lemma2}
Let  $u$ be a semiconcave  function on $\T^d$. Then $\W_u(x)\neq\emptyset$ for any $x\in\T^d$. Moreover, any  $\mu\in \W_u(x)$ is invariant under $\Xxu$.
\end{proposition}
In view of the above result and Proposition~\ref{pro:cnes}, any  $\mu\in \W_u(x)$ is supported by the critical set of $u$. Notice, however, that the portion of $\Crit(u)$ where $\mu$ concentrates depends on the initial point $x$.

\smallskip
We will now use  occupational measures to establish the main result of this section, showing that the critical set acts as an approximate attractor for the generalized gradient flow.
\begin{theorem}\label{the:asympt_critical}
Let  $u$ be a semiconcave  function on $\T^d$ and
 let $x\in\T^d$. Then for any $\eps>0$
\begin{equation}
\label{asympt_critical}
\lim_{T\to\infty}\frac {1}{T}
\LL^{1} \Big(  \big\{t \in [0,T]~:~d_{\Crit(u)}\big(\Xxu(t,x)\big)\geqslant\eps\big\}  \Big) =0, 
\end{equation}
where $\LL^{1}$ denotes the one-dimensional Lebesgue measure.
\end{theorem}
\proof Let us argue by contradiction assuming that there exist  numbers $\eps, \delta>0$ and a sequence $T_{k} \uparrow \infty$ such that
\begin{equation*}
\frac {1}{T_k}
\LL^{1} \big(  \big\{t \in [0,T_k]~:~d_{\Crit(u)}\big(\Xxu(t,x)\big)\geqslant\eps\big\}  \big) \geqslant\delta\qquad\forall k\in\N.
\end{equation*}
Without loss of generality we can suppose that $\mu_{x}^{T_{k}} \rightharpoonup \mu$ as $k\to \infty$ for some $\mu\in\W_u(x)$.
Therefore, for all $k\in\N$ we have that
\begin{multline}\label{eq;no_asympt_critical}
\int_{\T^{d}} d_{\Crit(u)}(y)\ d\mu^{T_k}_{x}(y) =\frac{1}{T_k}\int_0^{T_k}d_{\Crit(u)}\big(\Xxu(t,x)\big)dt
\\
\geqslant
\frac {\eps}{T_k}
\LL^{1} \big(  \big\{t \in [0,T_k]~:~d_{\Crit(u)}\big(\Xxu(t,x)\big)\geqslant\eps\big\}  \big) \geqslant\eps\delta.
\end{multline}
On the other hand, since $\mu$ has support in the closed set $\Crit(u)$ by Proposition~\ref{pro:cnes}, 
\begin{equation*}
\lim_{k\to\infty}\int_{\T^{d}} d_{\Crit(u)}(y)\ d\mu^{T_k}_{x}(y)=\int_{\T^{d}} d_{\Crit(u)}(y)\ d\mu(y)=0,
\end{equation*}
in contrast with \eqref{eq;no_asympt_critical}. The conclusion follows.
\qed

\smallskip
A  finer analysis of the asymptotic behavior  of $\Xxu$,  for  solutions of Hamilton-Jacobi equations, will be presented in Section~\ref{sec:occupational}.
Now, we  show that $\W_u$ remains constant along $\Xxu$.
\begin{theorem}\label{the:stability}
Let  $u$ be a semiconcave  function on $\T^d$.
Then for any $x\in\T^d$ we have that
\begin{equation}
\label{eq:stability}
\W_u(x)=\W_u\big(\Xxu(t,x)\big)\qquad\forall t\geqslant 0.
\end{equation}
\end{theorem}
\proof
Fix $t\geqslant 0$. Let  $\mu\in\W_u(x)$ and let $T_{k} \uparrow \infty$ be such that $\mu_{x}^{T_{k}} \rightharpoonup \mu$ as $k\to \infty$. For all $k$ large enough, so that $T_k>t$, let us set $S_k=T_k-t$.  Then, by the semigroup property of the flow we have that, for all $ f \in C(\T^{d})$, 
\begin{eqnarray*}
\lefteqn{\int_{\T^{d}} f(y)\ d\mu^{S_k}_{\Xxu(t,x)}(y) =\frac{1}{S_k} \int_{0}^{S_k} f\big(\Xxu(s+t,x)\big)\ ds}
\\
& =&\frac{1}{S_k} \int_{t}^{T_k} f\big(\Xxu(s,x)\big)\ ds
\\
& =&\frac{1}{T_k-t} \int_{0}^{T_k} f\big(\Xxu(s,x)\big)\ ds-\frac{1}{S_k} \int_{0}^{t} f\big(\Xxu(s,x)\big)\ ds
\\
& =&\frac{T_k}{T_k-t}\int_{\T^{d}} f(y)\ d\mu^{T_k}_{x}(y) -\frac{1}{S_k} \int_{0}^{t} f\big(\Xxu(s,x)\big)\ ds \stackrel{k\to\infty}{\longrightarrow} \int_{\T^{d}} f(y)\ d\mu(y).
\end{eqnarray*}
Therefore, $\mu\in \W_u\big(\Xxu(t,x)\big)$ and so $\W_u(x)\subset\W_u\big(\Xxu(t,x)\big)$. Since the reverse inclusion can be proved by a similar argument, the proof is complete.
\qed

\smallskip
Next, we investigate the  case when occupational measures are Dirac.
\begin{theorem}\label{the:delta_lom}
Let $u$ be a semiconcave  function on $\T^d$.
Let $x, \overline x \in \T^{d}$ and let $\{T_{k}\}_{k \in \N}$ be a real sequence, with $T_{k} \uparrow \infty$. The following properties are equivalent:
\begin{itemize}
\item[(a)] $\mu_{x}^{T_{k}} \rightharpoonup \delta_{\overline x}$ \;as\; $k\to \infty$,
\item[(b)] for any $\eps>0$ 
\begin{equation}
\label{eq:delta_lom}
\lim_{k\to\infty}\frac {1}{T_k}
\LL^{1} \big(  \big\{t \in [0,T_{k}]~:~|\Xxu(t,x)-\overline x|\geqslant\eps\big\}  \big) =0.
\end{equation}
\end{itemize}
\end{theorem}
\proof For any $\eps>0$ and any $T>0$ let us set
\begin{equation*}
F_{\eps}(T)= 
\big\{t \in [0,T]~:~|\Xxu(t,x)-\overline x|\geqslant\eps\big\}
\quad\mbox{and}\quad
G_{\eps}(T)=[0,T]\setminus F_{\eps}(T).
\end{equation*}
\fbox{$(a)\Rightarrow(b)$} Fix any $\eps>0$. Then we have that
\begin{equation*}
\frac {1}{T_k}\int_0^{T_k}|\Xxu(t,x)-\overline x|dt
\geqslant \frac {\eps}{T_k}
\LL^{1} \big(F_{\eps}(T_k) \big) .
\end{equation*}
Since the left-hand side converges to $\int_{\T^d}|y-\overline x|d\delta_{\overline x}(y)=0$ as $k\to\infty$, \eqref{eq:delta_lom} follows.

\smallskip\noindent
\fbox{$(b)\Rightarrow(a)$} Let $f\in C(\T^d)$. Then for any $\eps>0$
\begin{multline*}
\frac {1}{T_k}\Big|\int_0^{T_k}\Big(f\big(\Xxu(t,x)\big)-f(\overline x)\Big)dt\Big|
\\
\leqslant
\frac {1}{T_k}\int_0^{T_k}\big|f\big(\Xxu(t,x)\big)-f(\overline x)\big|
\Big(\chi_{_{F_{\eps}(T_k)}}(t)+\chi_{_{G_{\eps}(T_k)}}(t)
\Big)dt
\\
\leqslant \frac {2\|f\|_\infty}{T_k}\LL^{1} \big(F_{\eps}(T_k) \big) +
\omega_f(\eps),
\end{multline*}
where $\omega_f(\eps)=\sup_{|y-z|\leqslant\eps}|f(y)-f(z)|$. Therefore, owing to \eqref{eq:delta_lom},
\begin{equation*}
\limsup_{k\to\infty}\frac {1}{T_k}\Big|\int_0^{T_k}\Big(f\big(\Xxu(t,x)\big)-f(\overline x)\Big)dt\Big|
\leqslant \omega_f(\eps).
\end{equation*}
Since $\eps$ is arbitrary and $\omega_f(\eps)\to 0$ as $\eps\to 0$, the conclusion follows.
\qed

\medskip
We  now characterize the case when $\W_u(x)$ is {\em atomic}, that is  $\W_u(x)=\{\delta_{\overline x}\}$ for some $\overline{x}$.
\begin{corollary}\label{cor:delta_lom}
Let $x, \overline x \in \T^{d}$. The following properties are equivalent:
\begin{itemize}
	\item[$(a)$] $\W_u(x)=\{\delta_{\overline x}\}$,
	\item[$(b)$] $\mu_{x}^{T} \rightharpoonup \delta_{\overline x}$\; as \;$T\to \infty$,
	\item[$(c)$] for any $\eps>0$
	\begin{align*}
		\lim_{T\to\infty}\frac {1}{T}
\LL^{1} \big(  \big\{t \in [0,T]~:~|\Xxu(t,x)-\overline x|\geqslant\eps\big\}  \big) =0.
	\end{align*}
\end{itemize}
\end{corollary}

\proof
The equivalence between $(a)$ and $(b)$ follows from the definition of $\W_u(x)$. The equivalence between $(b)$ and $(c)$ is a consequence of Theorem~\ref{the:delta_lom}.
\qed

\begin{remarks}\label{rem:delta_lom}
Notice that property $(c)$ above can be expressed saying that $\overline x$ is the {\em approximate limit} of $\Xxu(t,x)$ as $t\to\infty$. 
\end{remarks}

%
%
Theorem~\ref{the:asympt_critical} suggests that  a point  $\overline x\in\T^d$ satisfying any of the properties in Corollary~\ref{cor:delta_lom} should  be critical for $u$. This is the object of our next result.
\begin{theorem}\label{the:delta_asym}
Let  $u$ be a semiconcave  function on $\T^d$ and let $x, \overline x \in \T^{d}$. If $\W_u(x)=\{\delta_{\overline x}\}$, then $\overline x\in \Crit(u)$.
\end{theorem}
\proof
Fix any $\eps>0$. Since $ \W_u(x)=\{\delta_{\overline x}\}$, Corollary~\ref{cor:delta_lom} ensures that  
\begin{equation}
\label{eq:Geps}
\lim_{t\to\infty}\frac {1}{t}
\LL^{1} \big(  G_{\eps}(t) \big) =1\quad\mbox{where}\quad G_{\eps}(t)= \big\{s \in [0,t]~:~|\Xxu(s,x)-\overline x|<\eps\big\}.
\end{equation}
Consequently, there exists $t_{\eps}\geqslant 1/\eps$ such that $|\Xxu(t_{\eps},x)-\overline x|<\eps$. In fact, $t_{\eps}\in G_{\eps}(t)$ for all $t\geqslant t_{\eps}$.
Moreover, by \eqref{eq:dggf}, the semigroup property of the flow, and Fubini's theorem for any $T>0$ we have that
\begin{eqnarray*}
\frac{1}{T}  \int_{0}^{T} \Big(u\big(\Xxu(t,\Xxu(t_{\eps},x))\big)-u\big(\Xxu(t_{\eps},x)\big)\Big) dt&=&\frac{1}{T} \int_{0}^{T} dt\int_{0}^{t}\big|p_0(\Xxu(s+t_{\eps},x))\big|^2ds
\\
& =&\frac{1}{T} \int_{0}^{T} (T-s)\big|p_0(\Xxu(s+t_{\eps},x))\big|^2ds.
\end{eqnarray*}
Since $\W_u(x)=\W_u\big(\Xxu(t_{\eps},x)\big)$ by Theorem~\ref{the:stability}, our hypothesis forces $\mu^T_{\Xxu(t_{\eps},x)}\rightharpoonup \delta_{\overline x}$ as $T\to\infty$. So,
\begin{equation*}
\lim_{T\to\infty}\frac{1}{T} \int_{0}^{T} (T-s)\big|p_0(\Xxu(s+t_{\eps},x))\big|^2ds=u(\overline x)-u\big(\Xxu(t_{\eps},x)\big)
\leqslant \omega_u(\eps).
\end{equation*}
Therefore, there exists $T_{\eps}>0$ such that, for all $T\geqslant T_{\eps}$
\begin{equation*}
\frac{1}{T} \int_{0}^{T} (T-s)\big|p_0(\Xxu(s+t_{\eps},x))\big|^2ds
\leqslant \omega_u(\eps)+\eps.
\end{equation*}
Now, Lemma~\ref{lem:abstract_meas2} below applied to 
$f(s):=|p_0(\Xxu(s+t_{\eps},x))|^2$ yields
\begin{equation}
\label{eq:om}
\LL^{1} \big(  K_{\eps}(T) \big) \geqslant  \frac T2-1
\quad\mbox{for all}\quad 
T\geqslant T_{\eps},
\end{equation}
where
\begin{equation*}
 K_{\eps}(T)= \big\{s \in [0,T]~:~|p_0(\Xxu(s+t_{\eps},x))|^2\leqslant \omega_u(\eps)+\eps\big\}.
\end{equation*}

\smallskip
Next, we claim that there exists $S_{\eps}\geqslant \max\{t_{\eps},T_{\eps}\}$ such that
\begin{equation}
\label{eq:om_claim}
K_{\eps}(T)\cap \big(G_{\eps}(T)-t_{\eps}\big)
\neq \emptyset \quad\mbox{for all}\quad 
T\geqslant S_{\eps}
\end{equation}
where we set $G_{\eps}(T)-t_{\eps}=\{s-t_{\eps}~:~s\in G_{\eps}(T)\}$.

Indeed, since
\begin{equation*}
\LL^{1} \big(  (G_{\eps}(T)-t_{\eps})\cap [0,T] \big)
\geqslant \LL^{1} \big(G_{\eps}(T)\big)-t_{\eps},
\end{equation*}
we have that $K_{\eps}(T)$ and $ (G_{\eps}(T)-t_{\eps})\cap [0,T]$ are subsets of $[0,T]$ satisfying
\begin{equation*}
\LL^{1} \big(  K_{\eps}(T) \big) + \LL^{1} \big(  (G_{\eps}(T)-t_{\eps})\cap [0,T] \big)\geqslant  \frac T2+\LL^{1} \big(G_{\eps}(T)\big)-\big(1+t_{\eps}\big)
\quad\forall
T\geqslant T_{\eps}.
\end{equation*}
Hence, our claim follows noting that, in view of \eqref{eq:Geps}, the right-hand side of the above inequality can be made strictly larger than $T$ for $T$ large enough.

\smallskip
Finally, appealing to \eqref{eq:om_claim} with $\eps=1/k$, one can construct a sequence 
\begin{equation*}
T_k:=s_{1/k}+t_{1/k}\to\infty\quad\mbox{as}\quad k\to\infty
\end{equation*}
with
\begin{equation*}
\Xxu(T_k,x)\to \overline x\quad\mbox{and}\quad p_0(\Xxu(T_k,x))\to 0 \quad\mbox{as}\quad k\to\infty.
\end{equation*}
It follows that $\overline x\in\Crit(u)$.
 \qed
\begin{lemma}\label{lem:abstract_meas2}
Let $T,C>0$ be fixed and let $f:[0,T]\to[0,\infty)$ be a  Lebesgue measurable function such that
\begin{equation}
\label{eq:abstract_meas2}
\frac 1T\int_0^T (T-s)f(s)ds\leqslant   C.
\end{equation}
Then we have that
\begin{equation}
\label{eq:abstract_meas2_conc}
\LL^{1} \big(  \big\{s \in [0,T]~:~f(s)\leqslant   C\big\}  \big) 
\geqslant \frac T2-1\,.
\end{equation}
\end{lemma}
\proof
Observe that \eqref{eq:abstract_meas2} yields
\begin{eqnarray*}
C&\geqslant& \frac 1T\int_0^T(T-s)f(s) \big(\chi_{_{\{f>   C\}}}(s)+\chi_{_{\{f\leqslant     C\}}}(s)\big)dr
\\
&\geqslant &\frac {  C}T\int_0^T(T-s)\chi_{_{\{f>   C\}}}(s)ds=\frac {  C}T\Big(\frac{T^2}2-\int_0^T(T-s)\chi_{_{\{f\leqslant     C\}}}(s)ds\Big)
\\
&\geqslant &\frac {  C}T\Big(\frac{T^2}2-T\LL^{1} \big(  \big\{r \in [0,T]~:~f(s)\leqslant    C\big\}  \big)\Big)
\end{eqnarray*}
The conclusion follows.
\qed

\smallskip
From Theorem~\ref{the:delta_asym} we can easily derive a characterisation of  the critical set. 
\begin{corollary}\label{cor:om}
Let  $u$ be a semiconcave  function on $\T^d$ and 
let $x\in\T^d$. Then
\begin{enumerate}
\item[(a)] $x\in \Crit (u)$ if and only if $ \W_u(x)=\{\delta_x\}$.
\end{enumerate} 
Moreover, for any $t\geqslant 0$,
\begin{enumerate}
\item[(b)] $\Xxu(t,x)\in \Crit (u)$ if and only if $ \W_u(x)=\{\delta_{\Xxu(t,x)}\}.$
\end{enumerate} 
\end{corollary}
\proof

$(a)$: the fact that $\W_u(x)=\{\delta_x\}$ for all $x\in \Crit (u)$ is obvious since $\Xxu(t,x)\equiv x$ in this case. The converse is an immediate consequence of Theorem~\ref{the:delta_asym}. 

\smallskip\noindent
$(b)$: the conclusion follows 
by combining $(a)$ and Theorem~\ref{the:stability}.
\qed
\section{Gradient flows and Hamilton-Jacobi equations}
\label{sec:solutions}
Given  a 
 function $V\in C^2(\T^d)$,  consider the  Lagrangean
\begin{equation*}
L(x,v)=\frac{1}{2}|v|^{2} - V(x)\quad (x \in \T^{d}, v\in\R^d)
\end{equation*}
with the associated Hamiltonian  
\begin{equation}\label{eq:H}
H(x,p)=\frac{1}{2}|p|^{2} + V(x)\quad (x \in \T^{d}, p\in\R^d).
\end{equation}
It is well known (see \cite{Fathi1997_1}) that there exists a unique constant $\mane \in \R$ such that the Hamilton-Jacobi equation
\begin{equation}\label{critical}
H(x, Du(x)) = \mane \quad (x \in \T^{d})
\end{equation}
has a semiconcave viscosity solution. Such a constant is known as the {\it critical constant} of $H$ and \eqref{critical} is called the {\it critical equation}. We recall that 
a continuous function $u: \T^d \to \R$ is a viscosity sub-solution of \eqref{critical}  if and only if 
\begin{equation}\label{eq:sub}
H(x, p) \leq \mane \quad \forall x \in \T^{d}\,,\;\forall p\in D^+u(x),
\end{equation}
  a viscosity super-solution of \eqref{critical}  if and only if 
\begin{equation}\label{eq:super}
H(x, p) \geq \mane \quad \forall x \in \T^{d}\,,\;\forall p\in D^-u(x),
\end{equation}
and a viscosity solution of \eqref{critical}  if and only if \eqref{eq:sub} and \eqref{eq:super} are both satisfied. See  \cite{Crandall_Evans_Lions1984} for equivalent definitions of viscosity solutions and related properties.

\begin{remarks}\label{rmk:singularity}
Well-known properties of semiconcave functions, together with \eqref{eq:sub} and \eqref{eq:super}, provide a simple  criterion to detect the singularities of a semiconcave solution $u$ of equation \eqref{critical}: $x\in\Sing(u)$ if and only if
\begin{equation}
\label{eq:singularity}
\min_{p\in D^+u(x)}H(x,p)<0
\end{equation}
(see \cite[Proposition~5.4.3]{Cannarsa_Sinestrari_book}).
\end{remarks}


\smallskip
The following facts are well known in weak KAM theory (see, e.g., \cite{Fathi_book,Cannarsa_Sinestrari_book,DaviniFathiIturriagaZavidovique2018}).
\begin{lemma}\label{lemma0}
The critical constant of $H$ in \eqref{eq:H} is given by
\begin{equation}
\label{eq:maneV}
\maneV=\max_{x \in \T^{d}} V(x).
\end{equation}
Moreover, 
for  any semiconcave  solution
 $u: \T^{d} \to \R$ of 
\begin{equation}
\label{eq:MS}
\frac{1}{2}|Du(x)|^{2} + V(x)=\maneV \quad (x \in \T^{d}).
\end{equation}
we have that 
\begin{equation}
\label{eq:argmax}
\argmax_{x\in \T^{d}}  V(x) = {\Crit}^{*}(u)
\end{equation}	
where $\Crit^{*}(u) = \Crit(u) \backslash \Sing(u)$ denotes the regular critical set of $u$. 
\end{lemma}
In weak KAM theory, the set in \eqref{eq:argmax} is called the {\em projected Aubry set}. 

\smallskip
For  solutions of \eqref{eq:MS} we  can develop a more detailed analysis of  the long-time behavior of the associated generalized gradient flow $\Xxu$ than the one given by Theorem~\ref{the:asympt_critical}. 
A fundamental role in such an analysis will be played by  singularities of solutions. In particular, it becomes essential to know whether $\Sing(u)$ is invariant under the generalized gradient flow. One usually refers to such an invariance property as the global propagation of singularities. 

In order simplify the notation and without loss of generality, hereafter {\em we will include the critical constant $\maneV$  in the potential $V$}. Thus, we will be concerned with  semiconcave solutions $u$  of the stationary equation 
\begin{equation}\label{HJ}
	\frac 12|Du(x)|^2+V(x)=0 \qquad(x\in\T^d).
\end{equation}
For any fixed $x\in\T^d$, we recall that $\Xxu(t, x)$ is the semi-flow associated with the  differential inclusion
\begin{equation}\label{GC}
\begin{cases}
\dXx(t, x) \in D^{+}u(\Xx(t, x)), \quad t\in[0,\infty)\mbox{ a.e.}
\\
\Xx(0, x)=x
\end{cases}
\end{equation}
While it has long been known that singularities propagate {\em locally} along  such a flow (\cite{Albano_Cannarsa2002}), the first global propagation result for the special case of the distance function on a Riemannian manifold was obtained in \cite{ACNS2013}. In \cite{Cannarsa_Mazzola_Sinestrari2015}, the global propagation of singularities was established for solutions of time dependent eikonal-type equations. Subsequently, in \cite{Albano2016_1}, the first author of this paper gave the invariance property for solutions of a class of time dependent Hamilton-Jacobi equations with variable coefficients. Applied to the present context, the main result of \cite{Albano2016_1} reads as follows:
\begin{theorem} \label{t:mono}
Let $V\in C^2(\T^d)$ and let $u$ be a semiconcave solution of \eqref{HJ}. If $x_0 \in \Sing(u)$, then 
\begin{equation}\label{SA}
\Xxu (t,x_0)\in \Sing (u)
\qquad\forall\,t\geqslant 0.
\end{equation} 
 \end{theorem}
Unfortunately, the proof of \cite{Albano2016_1} is incomplete. More precisely, 
the idea of the proof proposed in \cite{Albano2016_1} consists in showing that the Hamiltonian evaluated along an approximated gradient flow (that is, an {\em approximate energy}) satisfies a certain differential inequality. The conclusion should follow by passing to the limit and this last step contains a gap. In the present paper, not only do we  complete this argument but we introduce several simplifications of the original proof as well. Specifically:
\begin{itemize}
\item we make a different use of Alexandroff's theorem to avoid appealing to fundamental solutions,
\item we obtain the final weighted energy estimate by means of a non smooth Gronwall-type argument.
\end{itemize}
\smallskip

Before starting the fairly long proof of Theorem~\ref{t:mono}, let us introduce further notions which play a crucial role in the reasoning, namely the critical time function, the cut time function, and the cut locus. 

\smallskip
The {\em critical time function} for the generalized gradient flow of a semiconcave function  $u: \T^{d} \to \R$ is defined as
\[
\crit(x)=\inf\big\{t\geqslant 0~:~\Xxu(t,x)\in\Crit(u)\big\}\quad  (x\in\T^d)\,.
\]
Observe that $\crit(x)\in[0,\infty]$. The following proposition establishes a dichotomy that  applies to all the points with a finite critical time.
\begin{proposition}\label{pro:finite_critical}
 Let 
 $u: \T^{d} \to \R$ be a semiconcave  solution of \eqref{HJ} and let  $x\in\T^d$ be such that $\crit(x)<\infty$. Then the following holds true:
\begin{itemize}
	\item[$(a)$] $\Xxu(\crit(x),x)\in\Crit(u)$,
	\item[$(b)$] $V\big(\Xxu(\crit(x),x)\big)=0$ if and only if $\Xxu( \crit(x),x)\in \Crit^*(u)$,
	\item[$(c)$] $V\big(\Xxu(\crit(x),x)\big)<0$ if and only if $\Xxu( \crit(x),x)\in \Sing(u)$.
\end{itemize}
\end{proposition}
\proof
Point $(a)$ follows from the fact that $\Crit(u)$ is closed.
Point $(b)$  is a direct consequence of \eqref{eq:argmax}. In order to prove $(c)$, observe that, in view of  $(a)$ and $(b)$,
$$V\big(\Xxu(\crit(x),x)\big)<0\quad\Longleftrightarrow\quad \Xxu(\crit(x),x)\in\Crit(u)\setminus\Crit{^*}(u).
\eqno
\square$$
\begin{remarks}\label{rmk:finite_critical}
Notice that, if $\crit(x)<\infty$ and $V\big(\Xxu(\crit(x),x)\big)=0$, then $u$ is differentiable at $\Xxu(t,x)$ for all $t\geqslant 0$.
Indeed, $\Xxu(t,x)=\Xxu(\crit(x),x)\in \Crit^*(u)$ for all $t\geqslant\crit(x)$ by $(b)$ above. Moreover, there can be no time $t_0\in[0,\crit(x))$ such that $\Xxu(t_0,x)\in\Sing(u)$, for otherwise the semi-flow would  stay in the singular set for all $t\geqslant t_0$ by  Theorem~\ref{t:mono}. 
\end{remarks}

The \emph{cut time function} of $u$ is the map $\cut:\T^d\to[0,\infty]$ defined as
\[
\cut(x) = \sup \big\{ t \geqslant 0~:~\exists \, \gamma \in C^1([0, t], \T^d), \, \gamma(0) = x, \, u(\gamma(t)) - u(x) = A_t(x, \gamma(t)) \big\}\,,
\]
where $A_t(x,y)$ is the {\em action functional} given by
\begin{align*}
	A_t(x,y)=\inf_{\xi\in\Gamma^t_{x,y}}\int^t_0L(\xi(s),\dot{\xi}(s))\ ds,\quad t>0,\ x,y\in \T^d,
\end{align*}
with 
\begin{align*}
\Gamma^t_{x,y}=\{\xi:[0,t]\to \T^d~:~\,\xi \text{ is absolutely continuous,}\;\xi(0)=x,\; \xi(t)=y\}.
\end{align*}
\begin{remarks}\label{rmk:cut_set}
The cut time function may well be infinite. For example, if $x\in\Crit^*(u)$, then  owing to \eqref{eq:argmax} we have that the constant curve $\gamma(s)\equiv x$ 
satisfies 
\[
u(\gamma(t)) - u(x) = A_t(x, \gamma(t))=0\qquad\forall t\geqslant 0.
\] 
Therefore, $\cut(x)=\infty$ for all $x\in\Crit^*(u)$.
\end{remarks}
\begin{definition}\label{defn:cut_set}
The set 
\[
\Cut\,(u)=\{x\in \T^d~:~\cut(x)=0\}
\]
is called the \emph{cut locus} of $u$.
\end{definition}
\noindent
In view of Remark~\ref{rmk:cut_set} we have that 
\begin{equation}
\label{eq:crit_cut}
\Crit\!\!\,^*(u)\cap\Cut(u)=\varnothing.
\end{equation}
Another important property of the cut locus, obtained in \cite[Theorem~5.3]{Cannarsa_Cheng_Hong_Wang2024}, is that $\Cut(u)$ is invariant under the generalized gradient flow. In fact, we have that
\begin{equation}
\label{eq:cut_set}
\begin{cases}
 (a)
 &
 \Sing\,(u)\subset\Cut\,(u)
 \vspace{.1cm}
 \\
 (b)
 &
 x_0\in\Cut\,(u)\quad\Longrightarrow\quad \Xxu(t,x_0)\in\Cut\,(u)\quad \forall t\geqslant 0.
\end{cases}
\end{equation}
 
 \smallskip
 We can now give the proof of Theorem~\ref{t:mono}.
\begin{proof}[Proof of Theorem \ref{t:mono}]
We observe that the conclusion is trivial if $x_0\in\Crit(u)$. Moreover, for $x_0\not\in\Crit(u)$ it suffices to show that, for any $T_0\in [0,\crit(x_0))$, 
\begin{equation}
\label{eq:conclusion1}
\Xxu (t,x_0)\in \Sing (u)\qquad\forall t\in [0,T_0].
\end{equation}
Indeed, this  is obvious if $\crit(x_0)=\infty$. When $\crit(x_0)<\infty$, \eqref{eq:cut_set} and  Proposition~\ref{pro:finite_critical} ensure that
$
\Xxu\big(\crit(x_0),x_0\big)\in \Crit (u)\cap\Cut(u). 
$
Then $\Xxu\big(\crit(x_0),x_0\big)\in\Sing (u)$
by \eqref{eq:crit_cut}. 

We will prove \eqref{eq:conclusion1} arguing by contradiction. Suppose there exists   $T_0\in [0,\crit(x))$  such that 
\begin{equation}
\label{eq:assurdo}
\big\{ t\in[0,T_0]~:~\Xx(t, x_0 )\notin \Sing (u)\big\}\neq\varnothing
\end{equation}
and set
\begin{equation}
\label{eq:T}
T=\inf \{ t\in[0,T_0]~:~\Xx(t, x_0 )\notin \Sing (u)\}.
\end{equation}
Notice that, if $\Xx(T, x_0 )\in \Sing (u)$, then we immediately get a contradiction. Indeed,  either $T=T_0$ in contrast with \eqref{eq:assurdo}, or $T<T_0$ and  the aforementioned local result from \cite{Albano_Cannarsa2002} ensures the propagation of singularities on a right neighbourhood of $T$, contradicting \eqref{eq:T}.  
So, suppose $\Xx(T, x_0 )\notin \Sing (u)$ and define 
$$
H[t]=\min_{p\in D^+u(\Xx (t,x_0))} \Big\{\frac 12 |p|^2+V(\Xxu (t,x_0))\Big\}\,, \quad  t\in[0,T].
$$
For later use,  let us  observe that, 
without loss of generality, we may suppose that 
  \begin{equation}\label{-1}
  -1<H[t]\leqslant   0,\qquad \forall t\in [0 ,T].
  \end{equation}
 Indeed, $H[t]\leqslant 0$ because $u$ is a solution of  \eqref{HJ}. Moreover, since $u$ is Lipschitz and $V$ is bounded, $H[t]$ is bounded below. So, to ensure \eqref{-1} it sufficies to divide $H[t]$ by a sufficiently large constant.
More importantly,  on account of Remark~\ref{rmk:singularity} we have that 
\begin{equation}
\label{eq:assurdo2}
H[t]<0\qquad\forall t\in [0,T)
\end{equation}
because $ \Xx (t)\in \Sing (u)$ for all such $t$'s. 
Since we are assuming $\Xx(T, x_0 )\notin \Sing (u)$, we also have that $H[T]=0$. In fact,
the rest of the proof consists of showing that $H[T]<0$. 

\smallskip
Hereafter, we will  abbreviate the notation as follows:
$$
\begin{cases}
\Xx (t)=\Xxu (t,x_0)
\\ 
 V[t]=V(\Xxu (t,x_0))
 \\
 D V[t]=DV(\Xxu (t,x_0))
\\
\mathbf{p}[t]:=p_0( \Xxu (t,x_0))
\end{cases}
\qquad(t\in[0,T])
$$

\medskip
In order to bound the function $H[t]$ it seems  natural to differentiate equation \eqref{HJ} along the generalized gradient flow of $u$. For this purpose,
a preliminary formal computation shows that   $
D^2u(x)Du(x)+DV(x)=0. 
$
The next step of the proof of Theorem \ref{t:mono} consists of a rigorous derivation of the above identity, which is the object of the following lemma.
 \begin{lemma}\label{lem:D^2}
 Let $u$ be a semiconcave solution of \eqref{HJ}. Then
\begin{equation}\label{nucleo}
	D^2 u(x)Du(x)=-DV(x),\qquad x\in \T^d\quad \text{ a.e.} 
\end{equation}  
\end{lemma}
\begin{proof}
Let $x\in \T^d$ be a point where $u$ is twice differentiable and let $\xi$ be an arbitrary unit vector. Then, we have that 
\begin{equation}\label{eq:dis1}
\frac 12 |p_t|^2 +V(x+t\xi)\leqslant  0=\frac 12 |Du(x)|^2+V(x), 	 
\end{equation}
where $t>0$ and $p_t\in D^+u(x+t\xi)$ (in the first inequality we are using the fact that $u$ is a viscosity solution of \eqref{HJ}). 
Hence, the convexity inequality 
$$
|p_t|^2\geqslant   |Du(x)|^2+2 Du(x)\cdot (p_t-Du(x)), 
$$
and \eqref{eq:dis1} yield  
$$
\frac 1t Du(x)\cdot (p_t-Du(x))+\frac{V(x+t\xi)-V(x)}{t}\leqslant  0 . 
$$
Then, due to \eqref{eq:2diff}, we deduce that 
$$
D^2u(x)Du(x)\cdot \xi+\frac{V(x+t\xi)-V(x)}{t}\leqslant  o(1),
$$
as $t\to 0$. Taking the limit as $t\to 0$ in the formula above we get 
$$
[D^2u(x)Du(x)+DV(x)]\cdot \xi\leqslant  0, 
$$
and, since $\xi$ is an arbitrary unit vector, we conclude that \eqref{nucleo} holds. 
\end{proof}

\smallskip
We can now proceed with the proof of Theorem~\ref{t:mono}. Since 
$
0\notin D^+u( \Xx (t))
$
for every $t\in [0 , T]$,
 there exists an open subset of $\T^d$, say $\mathcal{U}$, such that   
$$
\mathcal{U}\cap\Crit(u)=\varnothing\quad\mbox{and}\quad \Xx (t)\in \mathcal{U},\quad \forall t\in [0,T].  
$$
Hence, we may assume that 
\begin{equation}\label{eq:signv} 
\forall x\in \mathcal{U}\qquad V(x)\leqslant  -M<0, 
	\end{equation} 
	for some $M>0$. 
	We define 
\begin{equation*}
W(x) = \frac{Du(x)}{\left |Du(x) \right |}\,,\quad
\Pi_x= W(x) \otimes W(x)\qquad   (x\in \mathcal{U}\mbox{ a.e.})
\end{equation*}
where we have used the standard notation 
$p \otimes q$ to denote the tensor product of two vectors $p,q\in\R^d$. 
Notice that $(p \otimes q)_{ij}=p_i q_j$ 
and $(p \otimes q)\,w= (q\cdot w)\,p$ for any $w\in\R^d$.
\begin{lemma}\label{lma:Bu}
For a.e.  $x\in \mathcal{U}$ we have  that  
\begin{equation}
\label{eq:upperbound2}
D^{2}u(x)\,  \xi \cdot \xi \leqslant  B_C(x)\; \xi \cdot \xi \qquad\forall \xi\in \R^d,
\end{equation}
with
\begin{multline*}
B_C(x)\\= C\Big(I+ \frac{Du(x) \otimes Du(x)}{ 2V(x)}\Big) + \frac{Du(x) \otimes D V(x)}{V(x)} 
+ \frac{Du(x) \cdot D V(x)}{ 4V^2(x)}\big(Du(x) \otimes Du(x)\big)\,,
\end{multline*}
where $C > 0$ is any semiconcavity constant for $u$.
\end{lemma}
\proof
Let $C > 0$ be a semiconcavity constant for $u$. For a.e. $x\in \mathcal{U}$ we have that
\begin{multline}\label{Eat}
D^{2}u(x)\,  \xi \cdot \xi =D^{2} u(x) \big(( I - \Pi_x) \xi + \Pi_x \xi\big)\cdot \big(( I - \Pi_x) \xi + \Pi_x \xi\big)
\\
= \underbrace{D^{2} u (x)( I - \Pi_x) \xi \cdot ( I - \Pi_x) \xi}_{(1)} + \underbrace{2 D^{2}u(x)( I - \Pi_x) \xi \cdot \Pi_x\xi}_{(2)} + \underbrace{D^{2}u(x) \Pi_x \xi \cdot \Pi_x \xi}_{(3)}, 
\end{multline}
for every $\xi \in \R^d$. 
Estimating each term of the sum, we obtain 
\begin{equation*}
(1) \leqslant   C |( I - \Pi_x ) \xi |^2 = C (I - \Pi_x )\xi\cdot\xi = C\big( | \xi |^2 - (W(x) \cdot \xi)^2\big).
\end{equation*}
Owing to Lemma~\ref{lem:D^2} we have 
\begin{align*}
(2) &= -2\;  \frac{W(x) \cdot \xi}{\left |Du(x)\right |}\;D V(x)\cdot \Big(\xi - \big(W(x) \cdot \xi\big)\,W(x)\Big)  
\\
&= -2 \;\frac{W(x) \cdot \xi}{\left |Du(x) \right |}\;\big(D V(x) \cdot \xi\big) + 2\; \frac{\big(W(x) \cdot \xi\big)^2}{\left |Du(x) \right |}\;\big( D V(x) \cdot W(x)\big)
\end{align*}
and
\begin{equation*}
(3) = - \; \frac{\big(W(x) \cdot \xi\big)^2}{\left |Du(x) \right |}\;\big( D V(x) \cdot W(x)\big) . 
\end{equation*}
Substituting the above estimates in \eqref{Eat} we deduce that
\begin{align*}
D&^{2}u(x) \, \xi \cdot \xi 
\\&\leqslant   C\big( | \xi |^2 - (W(x)\cdot \xi)^2\big)   -2 \;\frac{W(x) \cdot \xi}{\left |Du(x) \right |}\;\big(D V(x) \cdot \xi\big) +  \; \frac{\big(W(x) \cdot \xi\big)^2}{\left |Du(x) \right |}\;\big( D V(x) \cdot W(x)\big) .
\end{align*}
Observing that $|Du(x)|^2=-2V(x)$ for a.e. $x\in \T^d$ in view of  \eqref{HJ}, we have that
\begin{multline*}
C\Big(I- \frac{Du(x) \otimes Du(x)}{ |Du(x)|^2}\Big) -2 \frac{Du(x) \otimes D V(x)}{|Du(x)|^2} 
+ \frac{Du(x) \cdot D V(x)}{ |Du(x)|^4}\big(Du(x) \otimes Du(x)\big)
\\= C\Big(I+ \frac{Du(x) \otimes Du(x)}{ 2V(x)}\Big) + \frac{Du(x) \otimes D V(x)}{V(x)} 
+ \frac{Du(x) \cdot D V(x)}{ 4V^2(x)}\big(Du(x) \otimes Du(x)\big)\,.
\end{multline*}
The conclusion follows.
\qed
  
  \medskip
  Now, we need a smooth approximation of $H[t]$. For this purpose, we use a standard mollifier, that is,
  a nonnegative smooth function  $\eta $ supported in the unit ball $B_1(0)$, with $\int \eta \; dx =1$, and set 
  for any $m\in\N,\,m\geqslant1,$
  $$
  \eta_m(x)=m^d \eta (mx),\qquad x\in \R^d. 
  $$  
  Then, identifying $u$ with its periodic extension to $\R^d$, we define  
 \begin{equation*}
u_m (x):=( \eta_m* u) (x)\qquad(x\in\T^d)  .
 \end{equation*} 
We have that 
\begin{equation}
\label{eq:sup_um}
|Du_m(x)|\leqslant   Lip(u)\qquad(x\in \T^d),
\end{equation}
where $Lip(u)$ is the Lipschitz seminorm of $u$, as well as
\begin{equation}
\label{eq:conv_um}
u_m\to u \quad\mbox{uniformly on}\quad\T^d\,.
\end{equation}

Now, let $\Xx^m_t=\Xx^m(t)$ be the solution of the differential equation
\begin{equation*}
	\begin{cases}
		\dXx^m_t = Du_m(\Xx^m_t), \quad t\in[0,\infty)
		\\
		\Xx^m_0=x_0. 
	\end{cases}
\end{equation*}
By the Ascoli-Arzel\`a   theorem,  up to a subsequence, 
we have that $\Xx^m_t$ converges uniformly on n$[0,T]$ to a limit, say $\Xx^\infty_t$.  Furthermore, since $u$ is a semiconcave function, $\Xx^\infty_t$ is the unique solution of \eqref{GC}. Then,  $\Xx (t)=\Xx^\infty_t$.
Set  
\begin{equation*}
	H_m(t) = \frac{1}{2}|D u_m(\Xx^m_t)|^2  + V(\Xx^m_t), \quad t\in[0, T].
\end{equation*}
Then, we have that
\begin{equation}\label{diffineq0}
	\frac d{dt}H_m(t) =\;D^2 u_m(\Xx^m_t) D u_m(\Xx^m_t) \cdot D u_m(\Xx^m_t) + D V(\Xx^m_t)\cdot D u_m(\Xx^m_t).  
\end{equation}
We want to estimate the right-hand side of \eqref{diffineq0}.
\begin{lemma}\label{lma:lemma3}
For all $m\geqslant 1$ we have that
\begin{equation}
\label{eq:D2um}
D^2u_m(x)\leqslant   (\eta_m*B_C)(x)\qquad(x\in \T^d).
\end{equation}
\end{lemma}
\proof
Defining for every  $f:\T^d\to\R^d$,  $\xi \in \R^d$, and  $t>0$ 
\begin{equation*}
\tau_{t,\xi}f(x)=\frac{f(x+t\xi)-f(x)}t\qquad(x\in \T^d),
\end{equation*}
we have that
 \begin{equation*}
\tau_{t,\xi}Du_m(x)\cdot\xi=\int\eta_m(y)\big(\tau_{t,\xi}Du(x-y)\cdot \xi\big)\, dy\qquad(x\in \T^d).
\end{equation*}
 Observe that the semiconcavity of $u$ ensures that 
 \begin{equation*}
\eta_m(y)\big(\tau_{t,\xi}Du(x-y)\cdot \xi\big)\leqslant   C|\xi|^2\eta_m(y)\qquad(y\in \T^d\mbox{ a.e.})
\end{equation*}
where $C\geq0$ is any semiconcavity constant for $u$. Therefore, Fatou's lemma yields
\begin{align*}
D^2u_m(x)\xi\cdot\xi&=\lim_{t\downarrow 0}\tau_{t,\xi}Du_m(x)\cdot\xi
\leqslant   \int \limsup_{t\downarrow 0}\eta_m(y)\big(\tau_{t,\xi}Du(x-y)\cdot\xi\big)\,dy
\\
&= \int\eta_m(y)D^2u(x-y)\xi\cdot\xi\,dy\leqslant    (\eta_m*B_C)(x)\xi\cdot\xi\
\end{align*}
for all $x\in \T^d$, thus proving the conclusion \eqref{eq:D2um}.
\qed

\medskip
We now resume the proof of Theorem~\ref{t:mono}. From \eqref{diffineq0} and \eqref{eq:D2um} it follows that
\begin{equation}\label{diffineq}
	\frac d{dt}H_m(t) 
	\leqslant    \;(\eta_m * B_C)(\Xx^m_t) D u_m(\Xx^m_t) \cdot D u_m(\Xx^m_t) + D V(\Xx^m_t) \cdot D u_m(\Xx^m_t ).  
\end{equation} 
In order  to take the limit in \eqref{diffineq} we pass to an integral inequality. 
For any $t\in [0,T)$, by integrating between $t+h$ and $t$ the above estimate (with $0<h<T-t$) we get 
\begin{multline}\label{intineq}
	H_m(t+h)  \hspace{-0,5mm}-H_m(t) 
	\\
	\leqslant    \int_t^{t+h}\; \Big((\eta_m * B_C)(\Xx^m_s) D u_m(\Xx^m_s) \cdot D u_m(\Xx^m_s) + D V(\Xx^m_s) \cdot D u_m(\Xx^m_s)\Big)\; ds .   
\end{multline}

Our next step is the identification of the limit of $Du_m(\Xx_t^m)$ as $m\to\infty$.
\begin{lemma}
	For a.e. $t\in [0,T]$, 
	\begin{equation}\label{convvel}
\lim_{m\to \infty}		Du_m(\Xx_t^m)=\mathbf{p}[t]. 
	\end{equation}
\end{lemma}
\begin{proof}
We observe that, since $|Du_m (\Xx^m_t)|\leqslant   Lip (u)$,  there exists a subsequence, still denoted by $|Du_m(\Xx^m_t)|^2$, which weakly converges in $L^2$ to a function, say $\ell =\ell (t)$. Similarly, up to a subsequence, $Du_m(\Xx^m_t)$ weakly converges in $L^2$ to a map $q(t)$.     
For every $0\leqslant   t_0\leqslant   t_1\leqslant   T$, we have that 
	$$
	u_m(\Xx^m_{t_1})-	u_m(\Xx^m_{t_0})=\int_{t_0}^{t_1} |Du_m (\Xx^m_{t})|^2\; dt . 
	$$
	Taking the limit, as $m\to \infty$ in the above identity and using that $u_m$ and $\Xx^m_{t}$ converge uniformly to $u$ and $\Xx (t)$ respectively, we find 
	$$
\int_{t_0}^{t_1} \ell (t)\; dt=	\lim_{m\to \infty } \int_{t_0}^{t_1} |Du_m (\Xx^m_{t})|^2\; dt  =u(\Xx (t_1))- u(\Xx (t_0))=\int_{t_0}^{t_1} |\mathbf{p}[t]|^2\; dt . 
	$$ 
	
 We claim that $\ell (t)=| \mathbf{p}[t]|^2$, for a.e. $t\in [0,T]$. Indeed, due to the arbitrariness of $t_0<t_1$, we have that  
 $$
 F(t):=\int_0^t [\ell (s)-| \mathbf{p}[s]|^2]\; ds 
 $$
satisfies 
$$
F(t_1)=\int_0^{t_1}  [\ell (s)-| \mathbf{p}[s]|^2]\; ds =F(t_0)+\underbrace{\int_{t_0}^{t_1} [\ell (s)-| \mathbf{p}[s]|^2]\; ds }_{=0} ,  
$$	
i.e.,
$
F\equiv F(0)=0
$. 
Then, taking the derivative of $F$, we conclude that 
$$
\ell (t)=| \mathbf{p}[t]|^2,\qquad \text{for a.e. }t\in [0,T], 
$$
i.e. 
\begin{equation}\label{dump} 
\lim_{m\to\infty} |Du_m (\Xx^m_t)|^2= | \mathbf{p}[t]|^2,\qquad \text{for a.e. }t\in [0,T]. 
\end{equation} 

Similarly, we claim that, for a.e. $t\in [0,T]$, 
$
\mathbf{p}[t]=q(t), 
$
so that $Du_m(\Xx^m_t)$ converges weakly in $L^2$ to $\mathbf{p}[t]$. Indeed, for every $0\leqslant   t_0\leqslant   t_1\leqslant   T$, we have that 
	$$
\Xx^m_{t_1}-\Xx^m_{t_0}=\int_{t_0}^{t_1} Du_m (\Xx^m_{t})\; dt . 
	$$
	Taking the limit, as $m\to \infty$ in the above identity 
we have  that 
$$
\mathbf{x}(t_1)-\mathbf{x}(t_0)=\int_{t_0}^{t_1} q(t)\; dt 
$$
and, due to $\mathbf{x}(t_1)-\mathbf{x}(t_0)=\int_{t_0}^{t_1} \mathbf{p}[t]\; dt$ and the arbitrariness of $t_0$ and $t_1$, we deduce  our claim.  
 
Finally, since
$$
\int_0^T | Du_m(\Xx^m_t)-\mathbf{p}[t]|^2\; dt =\int_0^T\Big( | Du_m(\Xx^m_t)|^2 -2     Du_m(\Xx^m_t)\cdot  \mathbf{p}[t]+ |\mathbf{p}[t]|^2\Big)\; dt , 
$$
taking the limit in the above identity as $m\to \infty$, by \eqref{dump} and the fact that $Du_m(\Xx^m_t)$ weakly converges in $L^2$ to $\mathbf{p}[t]$  we get
$$
\lim_{m\to \infty} \int_0^T | Du_m(\Xx^m_t)-\mathbf{p}[t]|^2\; dt =2 \int_0^T\{ |\mathbf{p}[t]|^2 -   \mathbf{p}[t]\cdot  \mathbf{p}[t]\}\; dt =0.
$$
 So, we conclude that $\lim_{m\to \infty} Du_m(\Xx^m_t)=\mathbf{p}[t]$ for a.e. $t\in [0,T]$.   
\end{proof}

Now, observe that 
\begin{align*}
&(\eta_m * B_C) (\Xx^m_t) \; D u_m(\Xx^m_t) \cdot D u_m(\Xx^m_t)
\\
&= \int \eta_m (\Xx^m_t -y )
\Big \{ 
C\Big( | D u_m(\Xx^m_t) |^2+\frac{\big(Du(y) \cdot D u_m(\Xx^m_t) \big)^2 }{ 2V(y)}\Big) 
\\
&+ \frac{\big(Du(y) \cdot D u_m(\Xx^m_t)\big)\; \big(D V(y) \cdot D u_m(\Xx^m_t)\big)}{V(y)} 
+ \frac{\big(DV(y)\cdot Du(y)\big) \big(Du(y) \cdot D u_m(\Xx^m_t)\big)^2 }{ 4V^2(y)} \Big\}dy. 
\end {align*} 
 In order to bound the right-hand side of the above inequality we note that, since $V\in C^2(\T^d)$, in the formula above we may replace $V(y)$ with $V(\Xx^m_t )$ and $DV(y)$ with $DV(\Xx^m_t )$ up to an error which goes to zero uniformly w.r.t. $t$ as $m \to \infty$, i.e.,
  \begin{align*}
(\eta_m * B_C) (\Xx^m_t) & D u_m(\Xx^m_t) \cdot D u_m(\Xx^m_t)
\\
&=\int \eta_m (\Xx^m_t -y )
\Big \{ 
C\Big( | D u_m(\Xx^m_t) |^2+\frac{\big(Du(y) \cdot D u_m(\Xx^m_t) \big)^2 }{ 2V(\Xx^m_t )}\Big) 
\\
&\quad+ \frac{\big(Du(y) \cdot D u_m(\Xx^m_t)\big)\; \big(D V(\Xx^m_t ) \cdot D u_m(\Xx^m_t)\big)}{V(\Xx^m_t )} 
\\
&\quad+ \frac{\big(DV(\Xx^m_t )\cdot Du(y)\big) \big(Du(y) \cdot D u_m(\Xx^m_t)\big)^2 }{ 4V^2(\Xx^m_t )} \Big\}dy
+ \varepsilon_m
 \end {align*} 
 for all $t\in [0 ,T]$, with $\varepsilon_m\downarrow 0$ as $m\to\infty$.
 Furthermore, we claim that, in the third term of the above integral, we can replace $DV(\Xx^m_t )\cdot Du(y)$ with $DV(\Xx^m_t )\cdot \mathbf{p}[t]$  up to an error which goes to zero as $m\to \infty$. Indeed, first we observe that 
 \begin{multline*}
 \int \eta_m (\Xx^m_t -y  ) \frac{\big(DV(\Xx^m_t )\cdot (Du(y )-\mathbf{p}[t]) \big)\big(Du(y) \cdot D u_m(\Xx^m_t)\big)^2 }{ 4V^2(\Xx^m_t )}\, dy 
 \\
=  \int_{B_{\frac 1m} (\Xx^m_t)}\eta_m (\Xx^m_t -y  )  \frac{\big(DV(\Xx^m_t )\cdot (Du(y )-\mathbf{p}[t]) \big)\big(Du(y) \cdot D u_m(\Xx^m_t)\big)^2 }{ 4V^2(\Xx^m_t )}\, dy . 
 \end{multline*} 
 Then, for a.e. $y\in B_{\frac 1m} (\Xx^m_t)$, we have that 
 $$
	Du(y)\cdot Du_m (\Xx_t^m)= 	Du(y)\cdot \big(Du_m (\Xx_t^m)-\mathbf{p}[t] \big)
+Du(y)\cdot\mathbf{p}[t].
$$
In view of \eqref{convvel}, the first term on the right-hand side goes to zero for a.e. $t\in [0,T]$. 
Let us consider the second term: for a.e. $y\in B_{\frac 1m} (\Xx^m_t)$ we have that   
\begin{equation}\label{eq:duy}
Du(y)\cdot\mathbf{p}[t]= \big(Du(y)-\text{proj}_{D^+ u(\Xx (t))} (Du(y))\big) \cdot  \mathbf{p}[t] 
+ \text{proj}_{D^+ u(\Xx (t))} (Du(y)) \cdot  \mathbf{p}[t],
\end{equation}
where 
 $
 \text{proj}_{D^+ u(\Xx (t))} (Du(y))) 
 $
 stands for the projection of $Du(y)$ onto $D^+ u(\Xx (t))$. 
We claim that the first term on the right-hand side of  \eqref{eq:duy} goes to zero as $m\to \infty$. Indeed, in view of the upper semicontinuity of $D^+u$, we have that for every $\epsilon >0$ there exists $m_{\epsilon}$ such that, for every $m>m_{\epsilon}$ and for a.e. $y\in B_{1/m} (\Xx (t))$, 
$
| Du(y)-\text{proj}_{D^+ u(\Xx (t))} (Du(y)) |  <\epsilon .
$
Next, we observe that, on account of \eqref{const}, 
\begin{equation*}
\text{proj}_{D^+ u(\Xx (t))} (Du(y)) \cdot  \mathbf{p}[t]=|\mathbf{p}[t]|^2
\end{equation*}
for a.e. $y\in B_{\frac 1m} (\Xx^m_t)$  and a.e. $t\in [0,T]$.  

In conclusion, for a.e. $s\in[0,T]$ we have that
\begin{multline*}
\lim_{m\to \infty}  (\eta_m * B_C)(\Xx^m_s)\; D u_m(\Xx^m_s) \cdot D u_m(\Xx^m_s)
\\
=     
C\Big( |  \mathbf{p}[s] |^2+\frac{ |  \mathbf{p}[s] |^4 }{ 2V[s]}\Big) 
+ \frac{ |  \mathbf{p}[s] |^2 }{V[s]}\; D V[s] \cdot  \mathbf{p}[s]
+ \frac{ | \mathbf{p}[s]|^4 }{ 4V[s]^2}\; D V[s] \cdot  \mathbf{p}[s]. 
 \end{multline*}
So, by appealing to the dominated convergence theorem  we deduce that,  for all $t\in [0 ,T]$,
 \begin{multline*}
\lim_{m\to \infty} \int_{t}^{t+h} (\eta_m * B_C)(\Xx^m_s)\; D u_m(\Xx^m_s) \cdot D u_m(\Xx^m_s)\; ds  
\\
=\int_t^{t+h}    
\Big \{ 
C\Big( |  \mathbf{p}[s] |^2+\frac{ |  \mathbf{p}[s] |^4 }{ 2V[s]}\Big) 
+ \frac{ |  \mathbf{p}[s] |^2 }{V[s]}\; D V[s] \cdot  \mathbf{p}[s]
+ \frac{ | \mathbf{p}[s]|^4 }{ 4V[s]^2}\; D V[s] \cdot  \mathbf{p}[s]\Big \}\; ds  . 
 \end{multline*} 
 Hence, observing that,  since $|\mathbf{p}[t]|^2$ is right-continuous,
  $$
  \lim_{m\to \infty} \big (H_m(t+h)-H_m(t)\big )=H[t+h]- H[t],
  $$
we can take the limit as $m\to \infty$ in \eqref{intineq} to conclude that,  for all $t\in [0 ,T)$ and all $h\in(0,T-t)$,
\begin{multline*}
H[t+h]- H[t] \leqslant   \int_t^{t+h}    
\Big\{ 
C\Big( |  \mathbf{p}[s] |^2+\frac{ |  \mathbf{p}[s] |^4 }{ 2V[s]}\Big)
\\  
+\Big( \frac{ |  \mathbf{p}[s] |^2 }{V[s]}
+ \frac{| \mathbf{p}[s]|^4 }{ 4V[s ]^2}+1\Big)
\; \big(D V[s] \cdot  \mathbf{p}[s]\big)\Big\}\; ds 
\\
=\int_t^{t+h} \Big\{C|\mathbf{p}[s]|^2\, \frac{H[s]}{V[s]} +\big(D V[s] \cdot  \mathbf{p}[s]\big) \Big(\frac{ H[s]}{V[s]}\Big)^2 \Big\}\; ds . 
\end{multline*}
  Furthermore, since $u$ is a viscosity subsolution of \eqref{HJ}, we have that 
  \begin{align*}
  \frac 12 |\mathbf{p}[s]|^2 +V[s]\leqslant   0\iff \frac 12 \frac{ |\mathbf{p}[s]|^2}{V[s]}+1\geqslant0
 & \iff  \frac{ |\mathbf{p}[s]|^2}{V[s]}\geqslant -2 
  \\
  &\implies  C|\mathbf{p}[s]|^2\, \frac{H[s]}{V[s]} \underbrace{\le}_{H[s]\leqslant   0} -\,2\,C\,H[s].  
  \end{align*}
 Now, thanks to \eqref{-1} and \eqref{eq:signv}, we find that 
 $$
 -\frac{H[s]}{\;V[s]^2}\leqslant   \frac 1{M^2}, 
 $$
 and 
 $$
 \big(D V[s] \cdot  \mathbf{p}[s]\big) \Big(\frac{ H[s]}{V[s]}\Big)^2
 \leqslant  - \frac{\| DV\|_{L^\infty}\, Lip(u)}{M^2} \;H[s]
 $$
  Then,   
$$
  C|\mathbf{p}[s]|^2\, \frac{H[s]}{V[s]} +\big(D V[s] \cdot  \mathbf{p}[s]\big) \Big(\frac{ H[s]}{V[s]}\Big)^2 
\leqslant   \underbrace{\Big( -2C -\frac{\| DV\|_{L^\infty} \, Lip(u)}{M^2}\Big)}_{=:-C_1} \; H [s] . 
$$
So, for all $t\in [0 ,T)$ and all $h\in(0,T-t)$ we end up with the integral inequality 
\begin{equation}
\label{eq:dinin}
H[t+h]- H[t] \leqslant   -C_1\int_{t}^{t+h} H[s]\; ds
\end{equation}
with $C_1\geqslant 0$. 
We are now in a position to apply Lemma~\ref{lem:dinin} below which ensures that
%
%
$$
H[0]\geqslant e^{C_1t} H[t], \qquad \forall t\in  [0,T). 
$$
 Now, since we have assumed $u$ to be continuously differentiable at $\Xx(T, x_0 )$, we can take  the limit as $t\to T^-$ in the above inequality  to find the contradiction 
$
0= e^{C_1T}H [T] \leqslant   H[0]<0
$. 
This completes the proof of Theorem \ref{t:mono}. 
 \end{proof}
The following nonsmooth Gronwall-type lemma has just been used in the above proof.
\begin{lemma} \label{lem:dinin}
Let $\phi:[0,T]\longrightarrow (-\infty, 0)$ be a bounded right-continuous function such that 
\begin{equation}
\label{eq:dinigiusto}
\phi(t)-\phi(s) \leqslant   -C\int_{s}^{t} \phi(r)\; dr\qquad \forall t\in  [0,T)
\end{equation}
for some constant $C\geqslant 0$. Then, $[0,T)\ni t\mapsto e^{Ct} \phi(t)$ is decreasing. 
\end{lemma} 
\begin{proof}
Observe that, like any bounded right-continuous function, $\phi\in L^1(0,T)$. So,  fix any $t\in [0,T]$ and define
 $$
 \psi(s)=\int_s^t  \phi(r)\, dr \qquad s\in [0,t] . 
 $$
 Then $\psi \in AC ([0,t])$ and, owing to  \eqref{eq:dinin}, we have that
 $$
 \frac {d\psi}{ds}  (s)=-\phi(s)\leqslant -\phi(t)-C_1\psi(s)\qquad s\in[0,t] \text{ a.e.}
 $$
 or
 \begin{equation*}
 \frac {d}{ds} \Big(e^{C_1s}\psi(s)\Big)\leqslant - e^{C_1s}\phi(t) \qquad s\in[0,t] \text{ a.e.}
\end{equation*}
By integrating over $[s,t]$ we get
$
\big(e^{C_1(t-s)}-1\big)\,\phi(t)\leqslant C_1\,\psi(s)  
$
for all $s\in[0,t].$
So, returning to \eqref{eq:dinin}, we conclude that
\begin{equation*}
\phi(t)-\phi(s)\leqslant -C_1\,\psi(s)\leqslant \big(1-e^{C_1(t-s)}\big)\,\phi(t) \qquad \forall s\in[0,t],
\end{equation*}
which in turn ensures that $[0,t]\ni s\mapsto e^{C_1t} \phi(s)$ is decreasing. 
\end{proof}
\begin{remarks}\label{rmk:extension}
We note that the above proof applies, with minor changes, to the eikonal-type equation
\begin{equation}\label{eq:eikonal}
	\frac 12\langle \Lambda(x)Du(x),Du(x)\rangle +V(x)=\alpha[0]\qquad(x\in\T^d),
\end{equation}
where $V\in C^2(\T^d)$ and $\Lambda(x)$ is a positive definite symmetric $d\times d$ matrix, smoothly depending on $x\in\T^d$. In this case, which was addressed in \cite{Albano2016_1}, the role of \eqref{flow} is played by the
generalized characteristic system
\begin{equation}\label{eikonal_flow}
\begin{cases}
\dXx(t, x) \in \Lambda\big(\Xxu(t,x)\big)D^{+}u\big(\Xx(t, x)\big), \quad t \geqslant   0\mbox{ a.e.}
\\
\Xx(0, x)=x
\end{cases}
\end{equation}
which defines a Lipschitz semi-flow on $\T^d$. Consequently, instead of \eqref{eq:dggf}, one has that
\begin{equation}
\label{eq:eikonal_dggf}
\frac{d^+}{dt} u\big(\Xxu(t, x)\big) =
\big\langle \Lambda\big(\Xxu(t, x)\big)p_0^\Lambda\big(\Xxu(t, x)\big),p_0^\Lambda\big(\Xxu(t, x)\big)\big\rangle\qquad(t\geqslant 0),
\end{equation}
where
\begin{equation}
\label{eq:eikonal_ms}
p_{0}^\Lambda(y)= \argmin_{p \in D^{+}u(y)} \langle \Lambda(y)p,p\rangle\qquad(y\in\T^d).
\end{equation}
Moreover, the sets $\Crit(u), \Sing(u)$, $\Crit^{*}(u)$, and $\Cut(u)$ can be defined similarly, as well as  occupational measures. So, all the results of this paper extend to equation \eqref{eq:eikonal}, in particular Theorem~\ref{the:case=0}
and Corollary~\ref{cor:summa} below.
\end{remarks}

\section{Long-time behavior of gradient flows}
\label{sec:occupational}
In this section we apply the occupational-measure framework developed above to obtain a precise description of the asymptotic behavior of generalized characteristics in the critical case.
Given a semiconcave solution $u$ to \eqref{HJ}, we have already proved that:
\begin{itemize}
\item[$(a)$] $\Crit(u)$ is an approximate attractor for $\Xxu$ (Theorem~\ref{the:asympt_critical}),
\item[$(b)$] $\Sing(u)$ is invariant under $\Xxu$ (Theorem~\ref{t:mono}).
\end{itemize}
Moreover, we have fully characterized the behavior of the gradient flow at points $x\in\T^d$ such that $\crit(x)<\infty$
(Proposition~\ref{pro:finite_critical}).
We now proceed to study the flow at the points $x\in\T^d$ for which the critical time is infinite. Since
\begin{equation*}
\Crit(u)=\Crit\!\!\,^*(u)\cup \big(\Sing(u)\cap\Crit(u)\big)
\end{equation*}
we look for necessary and sufficient conditions for  $\Xxu$ to approach $\Crit^*(u)$ or the singular critical set of $u$ as $t\to\infty$. Our first result addresses $\Crit^*(u)$.
Let us set
\begin{equation}\label{eq:MV}
M(V)= {\Crit}^{*}(u)(=\argmax_{x\in \T^{d}}  V(x)).
\end{equation}
\begin{theorem}\label{the:case=0}
Let 
 $u: \T^{d} \to \R$ be a semiconcave  solution of \eqref{eq:MS} and let $x \in \T^{d}$ be such that $\crit(x)=\infty$. The  following properties are equivalent:
\begin{itemize} 
\item[(a)] for any $\mu\in\W_u(x)$ we have that
\begin{equation}
\label{eq:mean=0}
\int_{\T^{d}} V(y) d\mu(y) = 0,
\end{equation}
\item[(b)]  for any $\varepsilon>0$ we have that
\begin{equation}\label{eq:lim_mean=0}
\lim_{T\to\infty}\frac {1}{T}\LL^{1} \Big(  \big\{t \in [0,T]~:~d_{M(V)} \big(\Xxu(t, x)\big) \geqslant   \eps\big\}  \Big) =0. 
\end{equation}
\end{itemize}
\end{theorem}
\proof
Assume $(a)$  and suppose $(b)$ fails for some $\varepsilon>0$. Then, for some $\delta>0$ and some sequence  $T_{k} \uparrow \infty$,
we have that
\begin{equation}\label{eq:allT}
\frac {1}{T_k}\LL^{1} \Big(  \big\{t \in [0,T_k]~:~d_{M(V)} \big(\Xxu(t, x)\big) \geqslant   \eps\big\}  \Big) \geqslant   \delta.
\end{equation}
We can also assume that
$\mu_{x}^{T_{k}} \rightharpoonup \mu\in\W_u(x)$ as $k\to \infty$ without loss of generality. Set
\begin{equation*}
m(\eps)=\inf\big\{-V(y)~:~y\in\T^d\,,\;d_{M(V)}(y)\geqslant\eps\big\}
\end{equation*}
and observe that
$
0<m(\eps)\leqslant     2\|V\|_\infty.
$ 
Since for any $k\in\N$
\begin{equation*}
m(\eps)\LL^{1} \Big(  \big\{t \in [0,T_{k}]~:~d_{M(V)} \big(\Xxu(t, x)\big) \geqslant   \eps\big\}  \Big)\leqslant     
-\int_0^{T_k}V\big(\Xxu(t,x)\big)dt,
\end{equation*}
we conclude that
\begin{equation*}
\frac 1{T_k}\LL^{1} \Big(  \big\{t \in [0,T_{k}]~:~d_{M(V)} \big(\Xxu(t, x)\big) \geqslant   \eps\big\}  \Big)\leqslant     
\,-\,\frac 1{m(\eps)T_k}\, \int_0^{T_k}V\big(\Xxu(t,x)\big)dt.
\end{equation*}
Now, 
\begin{equation*}
\lim_{T\to\infty}\frac 1{T_k} \int_0^{T_k}V\big(\Xxu(t,x)\big)dt=\int_{\T^{d}} V(y) d\mu(y) = 0 
\end{equation*}
in view of \eqref{eq:mean=0}. This contradicts \eqref{eq:allT} showing that $(b)$ holds true.

\smallskip
Conversely, assume $(b)$. Let $\mu\in\W_u(x)$ and let $T_{k}\uparrow\infty$ be such that $\mu_{x}^{T_{k}} \rightharpoonup \mu$ as $k\to \infty.$
Fix any $\eps>0$. Then
\begin{multline*}
0\geqslant     \frac 1{T_{k}} \int_0^{T_{k}}V\big(\Xxu(t,x)\big)dt
\\
\geqslant     \min_{d_{M(V)}(y)<\eps}V(y)-\frac{2\|V\|_\infty}{T_k}\LL^{1} \Big(  \big\{t \in [0,T_{k}]~:~d_{M(V)} \big(\Xxu(t, x)\big) \geqslant   \eps\big\}  \Big) .
\end{multline*}
Therefore, taking the limit as $k\to\infty$ and recalling \eqref{eq:lim_mean=0} we obtain
\begin{equation*}
0\geqslant     \int_{\T^{d}} V(y) d\mu(y) \geqslant   \min_{d_{M(V)}(y)<\eps}V(y)\,,  
\end{equation*}
where the right-hand side of the above inequality tends to zero as $\eps\to 0$. Since $\eps$ is arbitrary the conclusion follows.
\qed
\begin{remarks}\label{rmk:case=0}
Owing to Corollary~\ref{cor:ggf}~$(i)$, we have that $\Xxu(\cdot,x)$ cannot approach $\argmin_{ \T^{d}}u$ unless $x$ is itself a minimum point of $u$. Therefore, if 
\begin{equation*}
x\in\T^d\setminus \argmin_{ \T^{d}}u\quad\mbox{and\quad $\Crit$}
^*(u)=\argmin_{ \T^{d}}u,
\end{equation*}
then there exists no  occupational measure of $\Xxu$ satisfying \eqref{eq:mean=0}.
\end{remarks}
We now consider the opposite situation.
\begin{theorem}\label{the:case<0}
Let 
 $u: \T^{d} \to \R$ be a semiconcave  solution of \eqref{eq:MS} and set  
 $$\delta(V):=\max_{\T^d}V-\min_{\T^d}V.$$ Let $x \in \T^{d}$ be such that $\crit(x)=\infty$. Then, the  following properties are equivalent:
\begin{enumerate}[(a)]
\item there exists $\mu\in\W_u(x)$ such that
\begin{equation}
\label{eq:mean<0}
\int_{\T^{d}} V(y) d\mu(y) < 0;
\end{equation}
\item $\delta(V)>0$ and  there exists a constant $\eta>0$  such that
\begin{equation}\label{eq:char_singularity}
 \limsup_{T\to\infty}
\frac{1}{T}\LL^{1} \Big(  \Big\{t \in [0,T]~:~\frac{1}{2}\big|p_0\big(\Xxu(t,x)\big)\big|^{2} + V\big(\Xxu(t,x)\big)+\eta\leqslant  0\Big\}  \Big) 
\geqslant \frac{\eta}{\delta(V)};
\end{equation}
\item there exists  a constant $\delta>0$ such that
\begin{equation}\label{eq:sing_density}
\limsup_{T\to\infty}\frac {1}{T}\LL^{1} \Big(  \big\{t \in [0,T]~:~d_{M(V)} \big(\Xxu(t, x)\big) \geqslant   \delta\big\}  \Big) >0
 \end{equation}
 where $M(V)$ is defined in \eqref{eq:MV}.
\end{enumerate}

\end{theorem}
\proof 
 $(a)\Rightarrow(b)$:   define 
\begin{equation}\label{eq:3_equiv}
\eta=-\frac 13 \int_{\T^{d}} V(y) d\mu(y)
\end{equation}
and observe that  $0<\eta\leqslant   \delta(V)/3$. Hence, $\delta(V)>0$. Let $T_{k} \uparrow \infty$ be such that $\mu_{x}^{T_{k}} \rightharpoonup \mu$ as $k\to \infty$. Since $u$ is bounded,  for  $k$  large enough, say $k\geqslant k_0$, we have that
\begin{multline*}
\frac 1{T_{k}} \int_0^{T_{k}}\Big(\frac{1}{2}\big|p_0\big(\Xxu(t,x)\big)\big|^{2} + V\big(\Xxu(t,x)\big)\Big)dt
\\
 =\frac{u\big(\Xxu(T_k,x)\big)-u(x)\big)}{2T_k} + \frac 1{T_{k}} \int_0^{T_{k}}\big( V\big(\Xxu(t,x)\big)\big)dt\leqslant   -2\eta.
\end{multline*}
So, 
by Lemma~\ref{lem:abstract_meas} below, applied to the function
$$
f(t)=-\frac{1}{2}\big|p_0\big(\Xxu(t,x)\big)\big|^{2} -V\big(\Xxu(t,x)\big)\qquad(t\in[0,T_k])
$$
with $ \delta=\delta(V),\,\rho=2\eta$, and $\lambda=\eta$, we conclude that
\begin{equation*}
\frac {1}{T_k}\LL^{1} \Big(  \Big\{t \in [0,T_{k}]~:~\frac{1}{2}\big|p_0\big(\Xxu(t,x)\big)\big|^{2} + V\big(\Xxu(t,x)\big)+\eta\leqslant  0\Big\}  \Big) 
\geqslant \frac{\eta}{\delta(V)-\eta}
\end{equation*} 
for all $k\geqslant k_0$. The conclusion \eqref{eq:char_singularity} follows. 

\medskip\noindent
$(b)\Rightarrow(c)$:
 let  $T_{k} \uparrow \infty$  be such that 
\begin{equation}\label{eq:c->d}
 \lim_{k\to\infty}
\frac{1}{T_k}\LL^{1} \Big(  \Big\{t \in [0,T_k]~:~\frac{1}{2}\big|p_0\big(\Xxu(t,x)\big)\big|^{2} + V\big(\Xxu(t,x)\big)+\eta\leqslant  0\Big\}  \Big) 
\geqslant \frac{\eta}{\delta(V)}.
\end{equation} 
In order to prove  \eqref{eq:sing_density} observe that, since $0=\max_{\T^d}V$, by the uniform continuity of $V$ we deduce that, for some $\delta=\delta(\eta)>0$,
\begin{equation*}
\big\{x\in\T^d~:~V(x)+\eta\leqslant   0\big\}\subset \big\{x\in\T^d~:~d_{M(V)}(x)\geqslant   \delta \big\}
\end{equation*}
So, for all $k\in \N$ we obtain the lower bound
\begin{multline*}
\frac{1}{T_k}\LL^{1} \Big(  \big\{t \in [0,T_k]~:~d_{M(V)} \big(\Xxu(t, x)\big) \geqslant   \delta\big\}  \Big)
\\
\geq
\frac{1}{T_k}\LL^{1} \Big(  \Big\{t \in [0,T_{k}]~:~V\big(\Xxu(t,x)\big)+\eta\leqslant  0\Big\}  \Big) 
\\
\geq
\frac{1}{T_k}\LL^{1} \Big(  \Big\{t \in [0,T_{k}]~:~\frac{1}{2}\big|p_0\big(\Xxu(t,x)\big)\big|^{2} + V\big(\Xxu(t,x)\big)+\eta\leqslant  0\Big\}  \Big) 
\end{multline*}
which in turn yields the conclusion in view of \eqref{eq:c->d}.


\medskip\noindent
$(c)\Rightarrow(a)$:  let us argue by contradiction assuming that \eqref{eq:mean=0} holds for all $\mu\in\W_u(x)$. Then, \eqref{eq:lim_mean=0} must also hold by Theorem~\ref{the:case=0}, 
in contrast with \eqref{eq:sing_density}.
\qed
\begin{lemma}\label{lem:abstract_meas}
Let $T,\delta>0$ be fixed and let $f:[0,T]\to[0,\delta]$ be a  Lebesgue measurable function such that
\begin{equation}
\label{eq:abstract_meas}
\frac 1T\int_0^Tf(t) dt\geqslant \rho
\end{equation}
for some $\rho\in(0,\delta)$. Then for any $\lambda\in(0,\rho)$ we have that
\begin{equation*}
\frac {1}{T}\LL^{1} \big(  \big\{t \in [0,T]~:~f(t)\geqslant\lambda \big\}  \big) 
\geqslant \frac{\rho-\lambda}{\delta-\lambda}.
\end{equation*}
\end{lemma}
\proof
Fix $\lambda\in(0,\rho)$ and let
\begin{equation*}
\varphi(\lambda)= \LL^{1} \big(  \big\{t \in [0,T]~:~f(t)\geqslant\lambda \big\}  \big)\qquad(t\in[0,T])
\end{equation*}
Then \eqref{eq:abstract_meas} ensures that
\begin{equation*}
\rho\leqslant   \frac 1T\int_0^Tf(t) \big(\chi_{_{\{f\geqslant \lambda\}}}(s)+\chi_{_{\{f< \lambda\}}}(s)\big)ds
\leqslant   \frac{\delta}{T}\varphi(\lambda)
+\frac{\lambda}{T}\big(T-\varphi(\lambda)\big).
\end{equation*}
Thus,
$
(\delta-\lambda)\varphi(\lambda)/T\geqslant \rho-\lambda$.
\qed

\medskip
The existence of a measure $\mu\in\W_u(x)$ satisfying the strict inequality \eqref{eq:mean<0} has the following consequence on propagation of singularities.
\begin{corollary}\label{cor:summa}
Let $x \in \T^{d}$. If \eqref{eq:mean<0} holds for some measure $\mu\in\W_u(x)$, then 
\begin{equation}\label{eq:sing_pro}
\exists\, t_0\in[0,\infty)\quad\mbox{such that}\quad \Xxu(t,x)\in\Sing(u)\quad\forall t\geqslant   t_0.
\end{equation}
\end{corollary}
\proof 
Suppose, first, $\crit(x)<\infty$. Then $\Xxu(\crit(x),x)\in\Crit(u)$ and $ \W_u(x)=\{\delta_{\Xxu(\crit(x),x)}\}$ by Proposition~\ref{pro:finite_critical} and Corollary~\ref{cor:om}.
Therefore, \eqref{eq:mean<0} reduces to $V\big(\Xxu(\crit(x),x)\big)<0$. So, again by Proposition~\ref{pro:finite_critical}, $\Xxu(\crit(x),x)\in\Sing(u)$. Thus,
$\Xxu(t,x)\in\Sing(u)$ for all $t\geqslant   \crit(x)$ owing to Theorem~\ref{t:mono}. This yields \eqref{eq:sing_pro} for $t_0=\crit(x)$.

Now, assume $\crit(x)=\infty$ and let $\eta>0$ be as in point $(b)$ of Theorem~\ref{the:case<0}.
Then,  \eqref{eq:c->d} holds for some sequence $T_{k} \uparrow \infty$.
Owing to Remark~\ref{rmk:singularity},
\begin{equation*}
  \Big\{t \in [0,T_{k}]~:~\frac{1}{2}\big|p_0\big(\Xxu(t,x)\big)\big|^{2} + V\big(\Xxu(t,x)\big)+\eta\leqslant  0\Big\}  \subset\Sing(u). 
\end{equation*}
Therefore,  for all $k\in\N$ we have that
\begin{equation*}
\frac {1}{T_k}\LL^{1} \big(  \big\{t \in [0,T_{k}]~:~\Xxu(t,x)\in\Sing(u)\big\}  \big) 
\geqslant \frac{\eta}{\delta(V)}
\end{equation*}
and \eqref{eq:sing_pro} follows by applying Theorem~\ref{t:mono} to  any 
$ t_0 \in\big\{t \in [0,T_{0}]~:~\Xxu(t,x)\in\Sing(u)\big\}  $. \qed
\begin{remarks}\label{rmk:case<0}
Proposition~\ref{pro:finite_critical}, Theorem~\ref{the:case=0}, and Corollary~\ref{cor:summa} provide the following synthetic view of the asymptotic behavior of $\Xxu(t,x)$ for any $x\in\T^ d$:
\begin{itemize}
\item either equality \eqref{eq:mean=0} holds for all measures $\mu\in\W_u(x)$, and  the regular critical set of $u$ is an approximate attractor for the generalized gradient flow of $u$,
\item  or the strict inequality \eqref{eq:mean<0} holds some  $\mu\in\W_u(x)$, and the flow enters the singular set of $u$ in finite time to remain there forever.
\end{itemize}
\end{remarks}

\section{Examples}
\label{sec:examples}
In this section, we discuss two examples to illustrate the interaction among occupational measures, critical points, and singularities as it has been described in this paper. 

In our first example, we see that the critical set has both a regular and a singular component. The generalized gradient flow starting outside $\Crit^*(u)$ always goes to the singular critical set reaching it in finite time.
\begin{example}
[One-dimensional pendulum]
\label{exa:pendulum}  Let $d=1$ and take
\begin{equation*}
V(x)=-\cos(2\pi x)\qquad(x\in\T^1).
\end{equation*}
Then $\maneV=1$ and the Hamilton-Jacobi equation \eqref{eq:MS} reads as follows
\begin{equation}
\label{eq:HJ_pendulum}
\frac 12|u'(x)|^2-\cos(2\pi x)=1\qquad(x\in\T^1).
\end{equation}
Observe that the periodic extension (with period $1$) of the function
\begin{equation*}
u(x)=
\begin{cases}\displaystyle
 -\int_0^x\sqrt{2\big(1+\cos(2\pi y)\big)}\,dy
 &
 0\leqslant   x<\frac12
 \vspace{.1cm}
 \\\displaystyle
 - \int_x^1\sqrt{2\big(1+\cos(2\pi y)\big)}\,dy
 &
 \frac 12\leqslant   x<1
\end{cases}
\end{equation*}
is a viscosity solution to \eqref{eq:HJ_pendulum}---in fact, unique up to constants. Moreover, $u$ is of class $\mathcal C^1$ on the open interval $(0,1)$ with
\begin{equation*}
u'(x)=
\begin{cases}\displaystyle
 -\sqrt{2\big(1+\cos(2\pi x)\big)}
 &
 0<   x\leqslant \frac12
 \vspace{.1cm}
 \\\displaystyle
 \sqrt{2\big(1+\cos(2\pi x)\big)}
 &
 \frac 12\leqslant   x<1
\end{cases}
\quad\mbox{and}\quad D^+u(0)=[-2,2]\,.
\end{equation*}
Consequently,
\begin{equation*}
\Sing(u)=\Z\,,\quad \Crit\!\!\,^*(u)=\Big\{\frac 12\Big\}+\Z\,,\quad \Crit(u)=\Big\{0,\frac 12\Big\}+\Z.
\end{equation*}
Moreover,  the gradient flow of $u$ satisfies
\begin{equation*}
\Xxu'(t,x)=
\begin{cases}
  -\sqrt{2\big(1+\cos(2\pi 
 \Xxu(t,x))\big)}
 &
 x\in (0,1/2)
 \vspace{.1cm}
 \\
 \sqrt{2\big(1+\cos(2\pi 
 \Xxu(t,x))\big)}
 &
 x\in(1/2,1)
\end{cases}
\end{equation*}
So, if $x\neq 1/2$ the flow $\Xxu(t,x)$ moves away from $\Crit\!\!\,^*(u)$---as observed in Remark~\ref{rmk:case=0}---and reaches $\Sing(u)$ in finite time.
For instance, for $x\in (0,1/2)$ the flow can  be computed by solving
\begin{equation*}
\begin{cases}
 \frac{\dXx(t)}{\sqrt{2\big(1+\cos(2\pi 
 \Xx(t))\big)}}
 =-1
 &
( t>0)
 \vspace{.1cm}
 \\
 \Xx(0)=x\,.
 &
\end{cases}
\end{equation*}
Since $0<x<1/2$, the above Cauchy problem can be recast as 
\begin{equation*}
\begin{cases}
 \frac{\dXx(t)}{2\cos(\pi 
 \Xx(t))\big)}
 =-1
 &
( t>0)
 \vspace{.1cm}
 \\
 \Xx(0)=x
 &
\end{cases}
\end{equation*}
to obtain, eventually,
\begin{equation}
\label{eq:soluzione_pendolo}
\Xxu(t,x)=\,\frac 2\pi\,\arctan \left[\frac{1+\tan \frac{\pi x}{2}-e^{2\pi t}\big(1-\tan\frac{\pi x}{2}\big)}{1+\tan \frac{\pi x}{2}+e^{2\pi t}\big(1-\tan\frac{\pi x}{2}\big)}\right]\quad(t\geqslant 0)
\end{equation}
Thanks to \eqref{eq:soluzione_pendolo}, we can compute explicitly the cut and critical time functions as follows
\begin{equation*}
\cut(x)=\crit(x)=\,\frac{1}{2\pi}\, \log\left(
\frac{1+\tan \frac{\pi x}{2}}{1-\tan \frac{\pi x}{2}}
\right)\qquad(0\leqslant x<1/2)
\end{equation*}
and deduce that $\W_u(x)=\left\{\delta_{0}\right\}$. The analysis of the case of $x\in (1/2,1]$ is totally similar. \hfill \qed
\end{example}

\smallskip
In our second example, the solution is smooth and the generalized gradient flow starting outside $\Crit^*(u)$ reaches the regular critical set in infinite time, as expected.
\begin{example}\label{exa:smooth}
 Let $d=1$ and take
\begin{align*}
	V(x)=-\sin^2(2\pi x)\qquad(x\in\T^1).
\end{align*}
Since $\alpha[0]=0$, the critical equation reduces to 
\begin{equation}\label{eq:HJ_degenerate}
	\frac 12|u'(x)|^2-\sin^2(2\pi x)=0\qquad(x\in\T^1).
\end{equation}
Observe that the projected Aubry set has two components, $\{0\}$ and $\{1/2\}$, and \eqref{eq:HJ_degenerate} admits the smooth solution
\begin{align*}
	u(x)=\frac{\sqrt{2}}{2\pi}\Big(1-\cos(2\pi x)\Big)\qquad(x\in\T^1)
\end{align*}
with the associated  gradient flow 
\begin{align*}
	\begin{cases}
		\dot{\mathbf{x}}(t)=\sqrt{2}\sin\big(2\pi\mathbf{x}(t)\big),\\
		\mathbf{x}(0)=x.
	\end{cases}
\end{align*}
Let us  consider the flow $\Xxu(\cdot,x)$ for $x\in(0,1/2)$. Then
\begin{align*}
	\bigg[\frac 1{2\pi}\log\big|\tan \big(\pi\Xxu(s,x)\big)\big|\bigg]^t_0=\int^t_0\frac{\Xxu'(s,x)}{\sin\big(2\pi\Xxu(s,x)\big)}\ ds=t\sqrt{2}
\end{align*}
which yields
\begin{align*}
	\pi\Xxu(t,x)=\arctan\big(e^{2\pi t\sqrt{2}}\tan (\pi x)\big).
\end{align*}
Thus, $\lim_{t\to+\infty}\Xxu(t,x)=1/2$,  $\crit(x)=\infty$, and $\W_u(x)=\left\{\delta_{1/2}\right\}$.
\hfill \qed
\end{example}

\bmhead{Acknowledgements}

Paolo Albano, Piermarco Cannarsa and Cristian Mendico were supported, in part, by the National Group for Mathematical Analysis, Probability and Applications of the Italian Istituto Nazionale di Alta Matematica "Francesco Severi''. Piermarco Cannarsa also acknowledges support from the Excellence Department Project awarded to the Department of Mathematics, University of Rome Tor Vergata, CUP E83C23000330006, and  the European Union---Next Generation EU, PRIN 2022 PNRR (CUP E53D23017910001). Wei Cheng is partly supported by National Natural Science Foundation of China (Grant No. 12231010). The authors also thank Jiahui Hong for checking the proof of the global propagation result. 

\section*{Declarations}


\begin{itemize}
\item Funding: National Group for Mathematical Analysis, Probability and Applications of the Italian Istituto Nazionale di Alta Matematica "Francesco Severi''; Excellence Department Project awarded to the Department of Mathematics, University of Rome Tor Vergata, CUP E83C23000330006, and  the European Union---Next Generation EU, PRIN 2022 PNRR (CUP E53D23017910001); National Natural Science Foundation of China (Grant No. 12231010).
\item The authors declare no conflict of interest. 
\item Ethics approval and consent to participate: Not applicable
\item Consent for publication
\item Data availability: Not applicable
\item Materials availability: Not applicable
\item Code availability: Not applicable
\item Author contribution: These authors contributed equally to this work
\end{itemize}

\bibliography{sn-bibliography}

@article{KrylovBogolyubov1937,
  author    = {N. N. Krylov and N. N. Bogolyubov},
  title     = {La th{\'e}orie g{\'e}n{\'e}rale de la mesure dans son application {\`a} l'{\'e}tude des syst{\`e}mes dynamiques de la m{\'e}canique non lin{\'e}aire},
  journal   = {Annals of Mathematics},
  volume    = {38},
  year      = {1937},
  pages     = {65--113},
  language  = {French}
}

@book{KatokHasselblatt,
  author    = {Anatole Katok and Boris Hasselblatt},
  title     = {Introduction to the Modern Theory of Dynamical Systems},
  publisher = {Cambridge University Press},
  year      = {1995}
}

@book{Petersen1983,
  author    = {Karl Petersen},
  title     = {Ergodic Theory},
  series    = {Cambridge Studies in Advanced Mathematics},
  volume    = {2},
  publisher = {Cambridge University Press},
  year      = {1983}
}

@article{DaviniFathiIturriagaZavidovique2018,
  author    = {Andrea Davini and Albert Fathi and Renato Iturriaga and Maxime Zavidovique},
  title     = {Convergence of the solutions of the discounted Hamilton--Jacobi equation: rate of convergence},
  journal   = {Transactions of the American Mathematical Society},
  volume    = {370},
  year      = {2018},
  pages     = {713--754}
}

@book{Aubin_Cellina1984,
 author = {Aubin, Jean-Pierre and Cellina, Arrigo},
 title = {Differential inclusions. {Set}-valued maps and viability theory},
 fseries = {Grundlehren der Mathematischen Wissenschaften},
 series = {Grundlehren Math. Wiss.},
 issn = {0072-7830},
 volume = {264},
 year = {1984},
 publisher = {Springer, Cham},
 language = {English},
 zbMATH = {3855514},
 Zbl = {0538.34007}
}

@Misc{Cannarsa_Cheng_Hong_Wang2024,
 Author = {Cannarsa, Piermarco and Cheng, Wei and Hong, Jiahui and Wang, Kaizhi},
 Title = {Variational construction of singular characteristics and propagation of singularities},
 Year = {2024},
 HowPublished = {Preprint, {arXiv}:2409.00961 [math.{AP}] (2024)},
 URL = {https://arxiv.org/abs/2409.00961},
 arXiv = {arXiv:2409.00961}
}

@Article{azagra-cappello-hajlasz_2023,
 Author = {Azagra, D. and Cappello, A. and Hajlasz, P.},
 Title = {A geometric approach to second-order differentiability of convex functions},
 FJournal = {Proceedings of the American Mathematical Society Ser. B},
 Journal = {Proc. Amer. Math. Soc. Ser. B},
 ISSN = { 2330-1511},
 Volume = {10},
 Number = {},
 Pages = {382--397},
 Year = {2023},
 Language = {English},
 zbMATH = {},
 Zbl = {7766042}
}

@article{Cannarsa_Cheng_Hong2023,
 author = {Cannarsa, Piermarco and Cheng, Wei and Hong, Jiahui},
 title = {Topological and control theoretic properties of {Hamilton}-{Jacobi} equations via {Lax}-{Oleinik} commutators},
 fjournal = {Nonlinear Analysis. Real World Applications},
 journal = {Nonlinear Anal., Real World Appl.},
 issn = {1468-1218},
 volume = {84},
 pages = {14},
 note = {Id/No 104282},
 year = {2025},
 language = {English},
 doi = {10.1016/j.nonrwa.2024.104282},
 zbMATH = {8005784},
 Zbl = {1560.35071}
}

@article{Chen_Hong_Zhao2022,
	author = {Chen, Cui and Hong, Jiahui and Zhao, Kai},
	doi = {10.3934/dcds.2021179},
	fjournal = {Discrete and Continuous Dynamical Systems. Series A},
	issn = {1078-0947},
	journal = {Discrete Contin. Dyn. Syst.},
	mrclass = {35F21 (49L12 49L25)},
	mrnumber = {4385783},
	number = {4},
	pages = {1949--1970},
	title = {Global propagation of singularities for discounted {H}amilton-{J}acobi equations},
	url = {https://mathscinet.ams.org/mathscinet-getitem?mr=4385783},
	volume = {42},
	year = {2022}}

@article{Artstein1999,
	author = {Artstein, Zvi},
	doi = {10.1006/jdeq.1998.3536},
	fjournal = {Journal of Differential Equations},
	issn = {0022-0396},
	journal = {J. Differential Equations},
	mrclass = {34A60 (34E15)},
	mrnumber = {1674533},
	mrreviewer = {Tzanko D. Donchev},
	number = {2},
	pages = {289--307},
	title = {Invariant measures of differential inclusions applied to singular perturbations},
	url = {https://mathscinet.ams.org/mathscinet-getitem?mr=1674533},
	volume = {152},
	year = {1999}}

@article{Cannarsa_Cheng2021b,
	author = {Cannarsa, Piermarco and Cheng, Wei},
	doi = {10.1007/s40574-021-00279-4},
	fjournal = {Bollettino dell'Unione Matematica Italiana},
	issn = {1972-6724},
	journal = {Boll. Unione Mat. Ital.},
	mrclass = {35F21 (37J50 49L25)},
	mrnumber = {4290348},
	number = {3},
	pages = {483--504},
	title = {Local singular characteristics on {$\Bbb R^2$}},
	url = {https://mathscinet.ams.org/mathscinet-getitem?mr=4290348},
	volume = {14},
	year = {2021}}

@article{Cannarsa_Cheng_Fathi2021,
	author = {Cannarsa, Piermarco and Cheng, Wei and Fathi, Albert},
	doi = {10.1007/s10240-021-00125-5},
	fjournal = {Publications Math\'{e}matiques. Institut de Hautes \'{E}tudes Scientifiques},
	issn = {0073-8301},
	journal = {Publ. Math. Inst. Hautes \'{E}tudes Sci.},
	mrclass = {Prelim},
	mrnumber = {4292741},
	number = {1},
	pages = {327--366},
	title = {Singularities of solutions of time dependent {H}amilton-{J}acobi equations. {A}pplications to {R}iemannian geometry},
	url = {https://mathscinet.ams.org/mathscinet-getitem?mr=4292741},
	volume = {133},
	year = {2021}}

@article{MR2237158,
	author = {Davini, Andrea and Siconolfi, Antonio},
	doi = {10.1137/050621955},
	fjournal = {SIAM Journal on Mathematical Analysis},
	issn = {0036-1410},
	journal = {SIAM J. Math. Anal.},
	mrclass = {49L25 (35B40 35F20 49K40)},
	mrnumber = {2237158},
	mrreviewer = {Fabio Bagagiolo},
	number = {2},
	pages = {478--502},
	title = {A generalized dynamical approach to the large time behavior of solutions of {H}amilton-{J}acobi equations},
	url = {https://mathscinet.ams.org/mathscinet-getitem?mr=2237158},
	volume = {38},
	year = {2006}}

@article{Bogaevsky2002,
	author = {Bogaevsky, Ilya Aleksandrovich},
	doi = {10.1016/S0167-2789(02)00652-8},
	fjournal = {Physica D. Nonlinear Phenomena},
	issn = {0167-2789},
	journal = {Phys. D},
	mrclass = {58K50 (35Q35 35Q53 76B99)},
	mrnumber = {1945478},
	mrreviewer = {Vladimir V. Tchernov},
	number = {1-2},
	pages = {1--28},
	title = {Perestroikas of shocks and singularities of minimum functions},
	url = {https://mathscinet.ams.org/mathscinet-getitem?mr=1945478},
	volume = {173},
	year = {2002}}

@article{Cannarsa_Chen_Cheng2019,
	author = {Cannarsa, Piermarco and Chen, Qinbo and Cheng, Wei},
	doi = {10.1016/j.jde.2019.03.020},
	fjournal = {Journal of Differential Equations},
	issn = {0022-0396},
	journal = {J. Differential Equations},
	mrclass = {37J50 (35F21 49L25)},
	mrnumber = {3950488},
	number = {4},
	pages = {2448--2470},
	title = {Dynamic and asymptotic behavior of singularities of certain weak {KAM} solutions on the torus},
	url = {https://mathscinet.ams.org/mathscinet-getitem?mr=3950488},
	volume = {267},
	year = {2019}}

@article{ACNS2013,
	author = {Albano, Paolo and Cannarsa, Piermarco and Nguyen, Khai Tien and Sinestrari, Carlo},
	fjournal = {Mathematische Annalen},
	issn = {0025-5831},
	journal = {Math. Ann.},
	mrclass = {35D40 (26B25 35A20 35A21 49J52)},
	mrnumber = {3038120},
	mrreviewer = {Luigi Rodino},
	number = {1},
	pages = {23--43},
	title = {Singular gradient flow of the distance function and homotopy equivalence},
	url = {https://doi.org/10.1007/s00208-012-0835-8},
	volume = {356},
	year = {2013}}

@article{Albano2016_1,
	author = {Albano, Paolo},
	fjournal = {Journal of Mathematical Analysis and Applications},
	issn = {0022-247X},
	journal = {J. Math. Anal. Appl.},
	mrclass = {35F21 (35A21)},
	mrnumber = {3535771},
	number = {2},
	pages = {1462--1478},
	title = {Global propagation of singularities for solutions of {H}amilton-{J}acobi equations},
	url = {https://doi.org/10.1016/j.jmaa.2016.07.037},
	volume = {444},
	year = {2016}}

@article{Albano_Cannarsa1999,
	author = {Albano, Paolo and Cannarsa, Piermarco},
	fjournal = {Annali della Scuola Normale Superiore di Pisa. Classe di Scienze. Serie IV},
	issn = {0391-173X},
	journal = {Ann. Scuola Norm. Sup. Pisa Cl. Sci. (4)},
	mrclass = {26B25 (41A50 49L20)},
	mrnumber = {1760538},
	mrreviewer = {Lud\v ek Zaj\'\i \v cek},
	number = {4},
	pages = {719--740},
	title = {Structural properties of singularities of semiconcave functions},
	url = {http://www.numdam.org/item?id=ASNSP_1999_4_28_4_719_0},
	volume = {28},
	year = {1999}}

@article{Albano_Cannarsa2002,
	author = {Albano, Paolo and Cannarsa, Piermarco},
	fjournal = {Archive for Rational Mechanics and Analysis},
	issn = {0003-9527},
	journal = {Arch. Ration. Mech. Anal.},
	mrclass = {35A20 (35F20)},
	mrnumber = {1892229},
	mrreviewer = {Luigi Rodino},
	number = {1},
	pages = {1--23},
	title = {Propagation of singularities for solutions of nonlinear first order partial differential equations},
	url = {https://doi.org/10.1007/s002050100176},
	volume = {162},
	year = {2002}}

@article{Cannarsa_Cheng3,
	author = {Cannarsa, Piermarco and Cheng, Wei},
	fjournal = {Calculus of Variations and Partial Differential Equations},
	issn = {0944-2669},
	journal = {Calc. Var. Partial Differential Equations},
	mrclass = {35F21 (37J50 49L25)},
	mrnumber = {3687884},
	number = {5},
	pages = {Art. 125, 31},
	title = {Generalized characteristics and {L}ax-{O}leinik operators: global theory},
	url = {https://doi.org/10.1007/s00526-017-1219-4},
	volume = {56},
	year = {2017}}

@article{Cannarsa_Mazzola_Sinestrari2015,
	author = {Cannarsa, Piermarco and Mazzola, Marco and Sinestrari, Carlo},
	fjournal = {Discrete and Continuous Dynamical Systems. Series A},
	issn = {1078-0947},
	journal = {Discrete Contin. Dyn. Syst.},
	mrclass = {35F21 (35A21 35F25 49J52)},
	mrnumber = {3392624},
	number = {9},
	pages = {4225--4239},
	title = {Global propagation of singularities for time dependent {H}amilton-{J}acobi equations},
	url = {https://doi.org/10.3934/dcds.2015.35.4225},
	volume = {35},
	year = {2015}}

@book{Cannarsa_Sinestrari_book,
	author = {Cannarsa, Piermarco and Sinestrari, Carlo},
	isbn = {0-8176-4084-3},
	mrclass = {49-02 (35F20 49K20 49L20)},
	mrnumber = {2041617},
	mrreviewer = {Pierre Cardaliaguet},
	pages = {xiv+304},
	publisher = {Birkh{\"a}user Boston, Inc., Boston, MA},
	series = {Progress in Nonlinear Differential Equations and their Applications},
	title = {Semiconcave functions, {H}amilton-{J}acobi equations, and optimal control},
	volume = {58},
	year = {2004}}

@article{Cannarsa_Yu2009,
	author = {Cannarsa, Piermarco and Yu, Yifeng},
	fjournal = {Journal of the European Mathematical Society (JEMS)},
	issn = {1435-9855},
	journal = {J. Eur. Math. Soc. (JEMS)},
	mrclass = {49K20 (26B25 35F21 49L20 49L25)},
	mrnumber = {2538498},
	mrreviewer = {Pietro Celada},
	number = {5},
	pages = {999--1024},
	title = {Singular dynamics for semiconcave functions},
	url = {https://doi.org/10.4171/JEMS/173},
	volume = {11},
	year = {2009}}

@article{Crandall_Evans_Lions1984,
	author = {Crandall, Michael G. and Evans, Lawrence C. and Lions, Pierre-Louis},
	fjournal = {Transactions of the American Mathematical Society},
	issn = {0002-9947},
	journal = {Trans. Amer. Math. Soc.},
	mrclass = {35F20 (35L60)},
	mrnumber = {732102},
	number = {2},
	pages = {487--502},
	title = {Some properties of viscosity solutions of {H}amilton-{J}acobi equations},
	url = {https://doi.org/10.2307/1999247},
	volume = {282},
	year = {1984}}

@article{Fathi1997_1,
	author = {Fathi, Albert},
	fjournal = {Comptes Rendus de l'Acad{\'e}mie des Sciences. S{\'e}rie I. Math{\'e}matique},
	issn = {0764-4442},
	journal = {C. R. Acad. Sci. Paris S{\'e}r. I Math.},
	mrclass = {58F27 (58F05 70H35)},
	mrnumber = {1451248},
	mrreviewer = {Ugo C. Bessi},
	number = {9},
	pages = {1043--1046},
	title = {Th{\'e}or{\`e}me {KAM} faible et th{\'e}orie de {M}ather sur les syst{\`e}mes lagrangiens},
	url = {https://doi.org/10.1016/S0764-4442(97)87883-4},
	volume = {324},
	year = {1997}}

@unpublished{Fathi_book,
	author = {Fathi, Albert},
	note = {Cambridge University Press, Cambridge (to appear)},
	title = {{W}eak {KAM} theorem in {L}agrangian dynamics}}

@article{Hurley1995,
	author = {Hurley, Mike},
	fjournal = {Journal of Dynamics and Differential Equations},
	issn = {1040-7294},
	journal = {J. Dynam. Differential Equations},
	mrclass = {58F25 (34C35 54H20 58F12)},
	mrnumber = {1348735},
	mrreviewer = {Romeo F. Thomas},
	number = {3},
	pages = {437--456},
	title = {Chain recurrence, semiflows, and gradients},
	url = {https://doi.org/10.1007/BF02219371},
	volume = {7},
	year = {1995}}

@article{Rifford2008,
	author = {Rifford, Ludovic},
	fjournal = {Communications in Partial Differential Equations},
	issn = {0360-5302},
	journal = {Comm. Partial Differential Equations},
	mrclass = {49L25 (35B65 35F20)},
	mrnumber = {2398240},
	number = {1-3},
	pages = {517--559},
	title = {On viscosity solutions of certain {H}amilton-{J}acobi equations: regularity results and generalized {S}ard's theorems},
	url = {https://doi.org/10.1080/03605300701382522},
	volume = {33},
	year = {2008}}

@article{Yu2006,
	author = {Yu, Yifeng},
	fjournal = {Annali della Scuola Normale Superiore di Pisa. Classe di Scienze. Serie V},
	issn = {0391-173X},
	journal = {Ann. Sc. Norm. Super. Pisa Cl. Sci. (5)},
	mrclass = {35F20 (35A21)},
	mrnumber = {2297718},
	number = {4},
	pages = {439--444},
	title = {A simple proof of the propagation of singularities for solutions of {H}amilton-{J}acobi equations},
	volume = {5},
	year = {2006}}
\begin{table}[h]
    \caption{Notation}
    \begin{tabularx}{\textwidth}{p{0.22\textwidth}X}
    \toprule
    $\chi_{_A}$ &  the characteristic function of a set $A\subset\R$ ($=1$ on $A$, $=0$ on $\R\setminus A$)\\
    $\T^ d$ & the flat $d$-dimensional torus\\
    $F:\T^d\rightrightarrows \R^d$ & a set-valued map from $\T^d$ into subsets of $\R^d$ \\
    $p \otimes q$& the tensor product of two vectors $p,q\in\R^d$\\
    $\displaystyle\omega_f$
    & the modulus of continuity of $f\in C(\T^d)$\\
    $\co S$ & the convex hull of the set $S$\\
    $\LL^{1}$ &  the one-dimensional Lebesgue measure\\
    $d_S(x)$ & the Euclidean distance of $x$ from the closed set $S$\\
   $D^+u(x)$ & the super-differential of $u$ at $x$\\
   $p_0(x)$ & the element of minimal norm of $D^+u(x)$\\
   $\Crit(u)$ & the critical set of  $u$\\
   $\Sing(u)$ & the singular set of  $u$\\
   $\Crit^*(u)$ & the regular critical set of   $u$\\
   $\Xxu(t,x)$ & the generalized gradient flow of $u$\\
   $\crit(x)$ & the critical time of $u$ at $x$\\
   $\cut(x)$ & the cut time  of   $u$  at $x$\\
   $\Cut(u)$ & the cut locus of   $u$\\
   $\mu^{T}_{x}$  & the (individual) occupational measure of $\Xxu(\cdot,x)$\\
   $\W_u(x)$ &  the family of all  occupational measures of $\Xxu(\cdot,x)$\\
        $\mane$ &  Ma\~n\'e's critical value\\
    $\delta_x$ & the Dirac measure centred at $x$\\
    \bottomrule
    \end{tabularx}
\end{table}
\end{document}